\documentclass[fleqn,11pt]{article}
\usepackage{tikz}
\usetikzlibrary{arrows,shapes,chains}
\usepackage{graphics}
\usepackage{color}
\usepackage{mathrsfs}
\usepackage{amssymb}
\usepackage{amsmath}

\begin{document}

\pagenumbering{arabic}
\newcounter{comp1}

\newtheorem{definition}{Definition}[section]
\newtheorem{proposition}{Proposition}[section]
\newtheorem{example}{Example}[section]
\newtheorem{method}{Method}[section]
\newtheorem{lemma}{Lemma}[section]
\newtheorem{theorem}{Theorem}[section]
\newtheorem{corollary}{Corollary}[section]
\newtheorem{application}{Application}[section]
\newtheorem{assumption}{Assumption}
\newtheorem{algorithm}{Algorithm}[section]
\newtheorem{remark}{Remark}[section]
\newcommand{\fig}[1]{\begin{figure}[hbt]
                  \vspace{1cm}
                  \begin{center}
                  \begin{picture}(15,10)(0,0)
                  \put(0,0){\line(1,0){15}}
                  \put(0,0){\line(0,1){10}}
                  \put(15,0){\line(0,1){10}}
                  \put(0,10){\line(1,0){15}}
                  \end{picture}
                  \end{center}
                  \vspace{.3cm}
                  \caption{#1}
                  \vspace{.5cm}
                  \end{figure}}
\newcommand{\Axk}{A(x^k)}
\newcommand{\Aumb}{\sum_{i=1}^{N}A_{i}u_{i}-b}
\newcommand{\Kk}{K^k}
\newcommand{\Kki}{K_{i}^{k}}
\newcommand{\Aukmb}{\sum_{i=1}^{N}A_{i}u_{i}^{k}-b}
\newcommand{\Au}{\sum_{i=1}^{N}A_{i}u_{i}}
\newcommand{\Aukpmb}{\sum_{i=1}^{N}A_{i}u_{i}^{k+1}-b}
\newcommand{\nab}{\nabla^2 f(x^k)}
\newcommand{\xk}{x^k}
\newcommand{\ubk}{\overline{u}^k}
\newcommand{\uhk}{\hat u^k}
\def\QEDclosed{\mbox{\rule[0pt]{1.3ex}{1.3ex}}} 
\def\QEDopen{{\setlength{\fboxsep}{0pt}\setlength{\fboxrule}{0.2pt}\fbox{\rule[0pt]{0pt}{1.3ex}\rule[0pt]{1.3ex}{0pt}}}}
\def\QED{\QEDopen} 
\def\proof{\par\noindent{\em Proof.}}
\def\endproof{\hfill $\Box$ \vskip 0.4cm}
\newcommand{\RR}{\mathbf R}

\title {\bf
A Variational Approach on Level sets and Linear Convergence of Variable Bregman Proximal Gradient Method for Nonconvex Optimization Problems}
\author{Daoli Zhu\thanks {Antai College of Economics and Management and Sino-US Global Logistics
Institute, Shanghai Jiao Tong University, Shanghai, China({\tt
dlzhu@sjtu.edu.cn})}, Sien Deng\thanks {Department of Mathematical Sciences, Northern Illinois University, DeKalb, IL, USA({\tt sdeng@niu.edu})}}
\footnotetext[1]{Acknowledgments: this research was supported by NSFC:71471112 and NSFC:71871140}

\maketitle

\begin{abstract}
We develop a new variational approach on level sets aiming towards convergence rate analysis of a variable Bregman proximal gradient (VBPG) method for a broad class of nonsmooth and nonconvex optimization problems. With this new approach, we are able to extend the concepts of Bregman proximal mapping and their corresponding Bregman proximal envelops, Bregman proximal gap function to  nonconvex setting. Properties of these mappings and functions are  carefully examined. An aim of  this work is to provide a solid foundation on which further design and analysis of  VBPG for
more general nonconvex optimization problems are possible. Another aim is to  provide a unified theory on linear convergence of VBPG with a particular interest towards proximal gradient methods. Central to our analysis for achieving the above goals is
an error bound in terms of level sets and subdifferentials (level-set subdifferential error bound)
  along with its links to other level-set error bounds.  As a consequence,
we have established a number of positive results. The newly established  results not only enable us to show that any accumulation of the sequence generated by VBPG is at least a critical point of the limiting subdifferential or even a critical point of the proximal subdifferential with a fixed Bregman function in each iteration, but  also provide a fresh perspective that allows us to explore   inner-connections among many known sufficient conditions for linear convergence of various first-order methods.  Along the way, we are able to derive a number of  verifiable conditions for level-set error bounds to hold, obtain  linear convergence of VBPG, and derive necessary conditions and sufficient conditions for
linear convergence relative to a level set for nonsmooth and nonconvex optimization problems.

\vspace{1cm}

%

\end{abstract}
\normalsize
\vspace{1cm}
\section{Introduction}
This paper studies the following nonconvex and nonsmooth optimization problem:
\begin{equation}\label{Prob:general-function}
\mbox{{\rm(P)}}\qquad\min_{x\in\RR^n}\qquad F(x)=f(x)+g(x)
\end{equation}
where $f:\RR^n\rightarrow(-\infty,\infty]$ is a proper lower semi-continuous (l.s.c) function that is smooth (may be nonconvex) in $\mathbf{dom}f$,  and $g:~\RR^n\rightarrow(-\infty, \infty]$ is a proper l.s.c  possible nonconvex and nonsmooth function. We say that (P) a convex problem (a fully nonconvex problem) if
both $f$ and $g$ are convex (both $f$ and $g$ are nonconvex).\\

Problem (P) arises naturally  in diverse areas such as compressed sensing~\cite{CandesTao05,Donoho06}, machine learning and statistics~\cite{Tibshirani1996}.  In such settings, $f$ can be viewed as  the data fitting part and
$g$  can be used to preserve structures such as sparsity, low-rankness, etc., to solutions of (P).  Typically these problems are of large scale. As a consequence, first-order methods and their enhanced versions are viewed to be a practical way to solve (P) with a huge number of decision variables~\cite{LionsMercier1979,Cohen80,Nesterov13}.\\

In this paper, by incorporating a Newton-like approach in each iteration, we propose to solve (P) by a general variable Bregman proximal gradient (VBPG) method. Based on the pioneering work ( Auxiliary Principle Problem) of Cohen~\cite{Cohen80}, an  iterative scheme of the VBPG method  for (P) can be stated as follows
\begin{equation}
x^{k+1}\in\arg\min_{x\in\RR^n}\bigg{\{}\langle\nabla f(x^k),x-x^k\rangle+g(x)+\frac{1}{\epsilon^k}D^k(x^k,x)\bigg{\}},
\end{equation}
where $D^k$ is a variable Bregman distance  (see Section 2.1 for the definition of a Bregman distance). The classical proximal gradient (PG) method is simply the choice of  $D^k(x,y)=\frac{1}{2}\|x-y\|^2$.
Other useful choices of variable proximal distance like functions can be obtained by choosing $D^k(x,x^k)=\frac{1}{2}\|x-x^k\|_{A^k}^2$ for Newton-like methods, where $A^k$ is an approximation of the Hessian $\nabla^2f(x^k)$ or a diagonal matrix. The second-order information through $D^k$ can be used to enhance the rate of  convergence  of the method. See \cite{Bonettini2016,Bonettini2016computing,Chouzenous2014} for details. Another choice is the Jacobi regularization method with
$$D^k(x,x^k)=\sum\limits_{i=1}^N [f\big{(}R_i^k(x)\big{)}+\frac{1}{2}\|x_i-x_i^k\|_{B_i}^2],$$ where $R_i^k\triangleq(x_1^{k},...,x_{i-1}^{k},x_i,x_{i+1}^k,...,x_N^k)$, and $B_i$ are positive definite matrices \cite{Banjac18}. The VBPG method can also be combined with extrapolation, proximal alternating linearization (see Algorithm 2.2 , Algorithm 3.1 of \cite{Cohen80}) and line search process (see  \cite{Zhu95}). The VBPG method for the general nonsmooth case is investigated in \cite{CohenZ}.\\

From a historical  and  broad view point, theory of error bounds (EB) has long been known playing an important role in
optimization theory~\cite{Rockafellar, Mordukhovich}, and a central role in the convergence and convergence rate analysis  of various iterative  methods~\cite{Pan97}. In fact well-known notions in variational analysis, such as  calmness, metric regularity, and submetric regularity to name a few, are all defined  in
terms of error bounds. As we are interested in finding an optimal solution, or a critical point, or an optimal value for (P),  it is very natural to look at the following  types of   error bounds: the first type EB is an inequality that bounds the {\em distance from a set of test points to
a target set} (e.g., critical-point set of (P), optimal solution set of (P), or a level set of $F$) by
a {\em residual function}; while the second type  EB is an inequality that bounds certain {\em absolute values of the difference between function  $F$ values at a set of test points and  a target value} (e.g., a critical value of $F$, or the optimal value of (P)) by a {\em residual function}.
These inequalities are evidently very useful in the convergence and rate convergence analysis of iterative optimization methods. In particular, when such inequalities are known with computable residual functions, they provide valuable quantitative information about the iterates by iterative optimization methods, and form a basis for convergence analysis,
convergence rate analysis, and finite termination criteria for these iterative methods.  In this regard, prominent examples in optimization of first type error bounds include  Hoffman's error bound~\cite{Hoffman1952}, sharp minimum~\cite{Cro78, Pol79},
weak sharp minima~\cite{BuF93, BuD02,StW99} (which will be termed level-set sharpness error bound in the latter sections of this paper), Robinson's error bound on polyhedral multifunctions~\cite{Rob81}.
Pioneering  contributions to second type error bounds include Polyak~\cite{Pol63} and {\L}ojasiewicz inequality~\cite{Loj63} although the latter inequality is not given in the context of optimization.\\

 An aim of this work is to build a comprehensive error-bound based mathematical theory
 for problem (P) with a particular interest  on applications to convergence rate analysis of VBPG.   As a general PG method is a special version of the VBPG method, a  brief review of literature on existing results on  convergence rate analysis of the PG method  is in order.

 It has been known that various first-order methods for (P) being a convex problem
exhibit a  converge rate of $O(1/k)$ or $O(1/k^2)$ sublinear rates~\cite{Bonettini2016,Bonettini2016computing,Nesterov13}. If  it is further assumed that $f$ is strongly convex and $g$ is convex, then it has been proved  that  PG methods can achieve a global linear convergence rate    in terms of  sequences of objective function values~\cite{Cohen17}. However,  strongly convexity is too restrictive to be satisfied in many practical problems.  To remediate this challenge and find weaker alternatives  that
are sufficient to obtain linear convergence of PG and acceleration techniques, \cite{Necoara2018} has introduced several relaxations of strong convexity, and \cite{Schopfer2016}~has  proposed a restricted scant inequality (RSI). \\

 Recently there is  a surge of interest in developing some first type error bound (EB) conditions that guarantee linear convergence for PG. A sample of such works includes Luo-Tseng  EB~\cite{Lewis2018, LuoTseng92}, quadratic growth condition~\cite{Lewis2018}, and metric subregularity condition~\cite{Lewis2018, YeJ18}.

As evident from the early work of \cite{Pol63},
second type EB conditions are also very useful.  A recent  success in optimization community is on various generalizations of the work
\cite{Pol63} with the aid of   the so called  Kurdyka-{\L}ojasiewicz (KL) property  to  obtain linear convergence of PG methods as well as a variety of other optimization methods~\cite{Attouch13,Chouzenous2014,Frankel2015}. ~\cite{Bolte2010} shows that the proximal gradient method has a linear convergence rate for functions satisfying KL property with exponent $\frac{1}{2}$; \cite{LiPong18} studies various calculus rules for KL exponents and illustrates that the Luo-Tseng error bound along with a  proper separation condition is sufficient for   $F$ being  a KL function with an exponent of $\frac{1}{2}$; \cite{Schmidt2016} proposes a proximal-PL inequality that leads to an elegant linear convergence rate analysis for sequences of function-values generated by the PG method. We remark that the proximal-PL inequality condition combines and extends an idea originated from  metric functions for variational inequalities (VI) by reformulating a VI as a constrained continuous differentiable optimization problem through certain gap functions, see~\cite{Zhu94,Zhu98}.\\

More recently,
there are two major lines of research on  error bound conditions to achieve linear convergence guarantee for gradient descent methods. The first line of research is to find  connections among existing error bound conditions. For the nonsmooth convex problem (P),  \cite{YeJ18} establishes the equivalence of metric subregularity, proximal error bound, KL property and quadratic growth conditions;~\cite{Schmidt2016} gives the equivalence of proximal-PL inequality with KL property and the proximal error bound condition. One of the important works in this regard is \cite{Lewis2018}, which studies the relationships among various EB conditions;
\cite{Zhang2019} introduces a set of abstract error bound conditions as well as  an abstract gradient  method.
Under the assumption of  (cor-res-EB) condition,  \cite{Zhang2019} establishes  that the (cor-EB) condition is a necessary and sufficient condition for linear convergence. In \cite{Zhang2019}, convex examples are given to illustrate these conditions.
As  nonconvex sparsity learning  problems have  received considerable attention in recent years,
another major line of  research is the study of (P) when (P) is fully nonconvex (by which we mean both $f$ and $g$
are nonconvex).  When the nonconvex sparsity-inducing penalties are introduced through the term $g$ of problem (P),  a challenge is how to develop efficient algorithms to solve these full nonconvex optimization problem with  large scale data.
\cite{Lewis2018} provides  new convergence analysis for the fully nonconvex composite optimization problem with the term $h(c(\cdot))$, where $c\in C^1$ (see Section 9 of \cite{Lewis2018}).~\cite{Ye18} uses a perturbation technique to study linear convergence  of the PG method for full nonconvex problem (P) under calmness as well as various equivalent conditions to calmness.\\

Motivated  by the above mentioned works for a quest for linear convergence of the PG method for (P), we are led to ask the following basic questions:
What are fundamental properties associated with
$F$ itself so that linear  convergence  of  VBPG is guaranteed? Could  a property possibly weaker than  K-L property
exist to ensure  linear convergence of VBPG? As well understood in variational analysis, properties of a function relate very naturally to the level sets  of the function. This leads us to look into  error bounds involving  level sets, subdifferentials  and  various  level-set  error bounds. A significant departure of our work to the above cited works is the use of  level sets as target sets (see the definition of first type error bounds)  to establish error bound conditions whereas the above cited works  typically use optimal solution sets or sets of critical points (in the  nonconvex case) as target sets to establish error bound conditions. The title of this work
reflects our level-set based perspective for this study.  As a result,  we have discovered a number of interesting  results  on level sets of
$F$,  revealed the roles of  level-set based  error bounds in establishing   linear convergence of VBPG, and uncovered  interconnections among level-set based  error bounds and other known error bounds  in  the literature.

A goal of this work is to
find the weakest possible conditions on $F$ in terms of error bounds  under which we are able to arrive  at  linear convergence of VBPG.  Specifically, such  conditions should  meet the following requirements
\begin{itemize}
\item[(i)] In the fully nonconvex setting (i.e., both $f$ and $g$ are nonconvex), the conditions  are sufficient  for $Q-$linear convergence of $\{F(x^k)\}$ and $R-$linear convergence of $\{x^k\}$ generated by VBPG. Moreover, all known sufficient conditions  for linear convergence
of PG methods  imply the conditions.

\item[(ii)] The conditions along with their associated theorems will provide a unique perspective that allows us to make  connections with many known conditions in the literature which are shown to guarantee the linear convergence of PG.
\end{itemize}

In this work, we provide an answer to the about questions in terms of error bounds involving level sets and subdifferentials and will illustrate that these error bound conditions have met the above requirements.
In addition to the above contributions, we also provide necessary conditions and sufficient conditions for linear convergence with respect to level sets for VBPG.
Furthermore, in Proposition~7.2, we have shown that the notion of level-set subdifferential EB condition is much weaker than that of the so called KL property.
To our knowledge, this work  is the first comprehensive work on  convergence rate analysis of VBPG. Moreover, a number of new results obtained in this work for VBPG are also new results even for the PG method. We believe  tools and concepts used in our analysis
for VBPG could be readily
adopted  for convergence rate analysis of more broad
classes of iterative optimization methods.

As this comprehensive work examines and establishes numerous error bound conditions and results,  we supply  Figure 1 in Section~ \ref{sec:level-set type error relationships} and Figure 2 in Section~\ref{sec:sufficient}
 to aid the reader to see easily inner relationships of these conditions and results.

The remaining part of this paper is structured as follows. Section~\ref{sec:pre} provides notation and preliminaries.  Section~\ref{sec:level-set} introduces level-set analysis, and studies level-set type error bounds. Section~\ref{sec:convergence} and section~\ref{sec:convergence rate} present the results on  convergence and linear convergence  analysis for VBPG respectively. Section 6 provides enhanced analysis under strong level-set error bounds and semiconvexity of $g$. Section \ref{sec:level-set type error relationships}
 investigates connections of various level-set  error bounds established in this work with other existing error bounds.
Section~\ref{sec:sufficient} lists known  sufficient conditions to guarantee the existence of level-set subdifferential error bounds. \\
\section{Notations and preliminaries}\label{sec:pre}
Throughout this paper, $\langle\cdot,\cdot\rangle$ and $\|\cdot\|$ denote the Euclidean scalar product of $\RR^n$ and its corresponding norm respectively. Let $\mathbf{C}$ be a subset of $\RR^n$ and $x$ be any point in $\RR^n$. Define
$$dist(x,\mathbf{C})=\inf\{\|x-z\|:z\in\mathbf{C}\}.$$
When $\mathbf{C}=\emptyset$, we set $dist(x,\mathbf{C})=\infty$.

The definitions we will use throughout the paper on subdifferential calculus are  standard in variational analysis (\cite{Rockafellar} and \cite{Mordukhovich}).
\begin{definition}[\cite{Rockafellar}]
Let $\psi$: $\RR^n\rightarrow\RR\cup\{+\infty\}$ be a proper lsc function.
\begin{itemize}
\item[{\rm(i)}] The domain of $\psi$, denoted by $\mathbf{dom}~ \psi$, is $\{x\in\RR^n:\psi(x)<+\infty\}$.
\item[{\rm(ii)}] For each $\overline{x}\in\mathbf{dom}~\psi$, the Fr\'echet subdifferential of $\psi$ at $\overline{x}$, written $\partial_{F}\psi(\overline{x})$, is the set of vectors $\xi\in\RR^n$, which satisfy
    $$\lim\limits_{\substack{y\neq x\\ y\rightarrow x}}\inf\frac{1}{\|x-y\|}[\psi(y)-\psi(\overline{x})-\langle\xi,y-x\rangle]\geq0.$$
    If $x\notin\mathbf{dom} \psi$, then $\partial_F\psi=\emptyset$.
\item[{\rm(iii)}] The limiting-subdifferential (\cite{Mordukhovich}), or simply the subdifferential for short, of $\psi$ at $\overline{x}\in\mathbf{dom}~\psi$, written $\partial_L\psi(\overline{x})$, is defined as follows:
    $$\partial_L\psi(\overline{x}):=\{\xi\in\RR^n:\exists x_n\rightarrow x, \psi(x_n)\rightarrow \psi(\overline{x}), \xi_n\in\partial_F\psi(x_n)\rightarrow\xi\}.$$
\item[{\rm(iv)}] The proximal subdifferential of $\psi$ at $\overline{x}\in\mathbf{dom}\psi$ written $\partial_P\psi(\overline{x})$, is defined as follows:
    $$\partial_P\psi(\overline{x}):=\{\xi\in\RR^n:\exists\rho>0, \eta>0\quad\mbox{s.t.}\quad \psi(x)\geq\psi(\overline{x})+\langle\xi,x-\overline{x}\rangle-\rho\|x-\overline{x}\|^2, \forall x\in\mathbb{B}(\overline{x};\eta)\}.$$
    where $\mathbb{B}(\overline{x}; \eta)$ is the open ball of radius $\eta>0$, centered at $\overline{x}$.
\end{itemize}
\end{definition}
\begin{definition}[\cite {Bacak10,Bernard05,Bolte2010,Rockafellar,Vial1983}]
Let $\psi:\RR^n\rightarrow(-\infty,\infty]$ be a proper lsc function.
\begin{itemize}
\item[{\rm(i)}]  (Definition~13.27 of \cite{Rockafellar}) A lsc function $\psi$ is said to be prox-regular at $\overline{x}\in \mathbf{dom}~\psi$ for subgradient $\overline{\nu}\in\partial_L\psi(\overline{x})$, if there exist parameters $\eta>0$ and $\rho\geq0$ such that for every point $(x,\nu)\in gph\partial\psi$ obeying $\|x-\overline{x}\|<\eta$, $|\psi(x)-\psi(\overline{x})|<\eta$, and $\|\nu-\overline{\nu}\|<\eta$ and $\nu\in\partial_L\psi(x)$, one has
    \begin{equation}
    \psi(x')\geq\psi(x)+\langle\nu,x'-x\rangle-\frac{\rho}{2}\|x'-x\|^2,\qquad\mbox{for all}\qquad x'\in\mathbb{B}(\overline{x};\eta)
    \end{equation}
\item[{\rm(ii)}] (Proposition~3.3 of \cite{Bernard05})~ A lsc function $\psi$ is said to be uniformly prox-regular around $\overline{x}\in \mathbf{dom}~\psi$ , if there exist parameters $\eta>0$ and $\rho\geq0$ such that for every point $x, x'\in \mathbb{B}(\overline{x};\eta)$ and $\nu\in\partial_L\psi(x)$, one has
    \begin{equation}\label{eq:varphi-0}
    \psi(x')\geq\psi(x)+\langle\nu,x'-x\rangle-\frac{\rho}{2}\|x'-x\|^2.
    \end{equation}
\item[{\rm(iii)}] (Definition~10 of \cite{Bolte2010})~A lsc function $\psi$ is semi-convex on $\mathbf{dom}~\psi$ with modulus $\rho>0$ if there exists a convex function $h:\RR^n\rightarrow\RR$ such that $\psi=h(x)-\frac{\rho}{2}\|x\|^2$, one has
\begin{equation}\label{eq:varphi}
\psi(y)\geq\psi(x)+\langle\xi,y-x\rangle-\frac{\rho}{2}\|x-y\|^2, \forall\xi\in\partial_L\psi(x).
\end{equation}
\item[{\rm(iv)}] (Theorem~10.33 of \cite{Rockafellar})~A lsc function $\psi$ is lower-$C^2$ on an open set $V$ if at any point $x$ in $V$, $\psi$ plus a convex quadratic function is a convex function on an open neighborhood $V'$ of $x$.
\end{itemize}
\end{definition}
It is well-known  that the following relations hold for a given proper lsc $\psi$.
\begin{eqnarray*}
\mbox{Convexity}\Rightarrow\mbox{Semi-convexity}\Rightarrow\mbox{Lower-$C^2$ on any open set}&\Rightarrow&\mbox{uniform prox-regularity around $x\in V$}\\
&\Rightarrow&\mbox{prox-regularity at all $x\in V$}
\end{eqnarray*}
Nonconvex regularization terms such as the smoothly clipped absolute derivation (SCAD)~\cite{Fan2001} and the minimax concave penalty (MCP)~\cite{Zhang2010} are examples of semi-convex functions.\\
For subdifferentials,  the following inclusions hold: $\partial_P\psi(x)\subset\partial_{F}\psi(x)\subset\partial_L\psi(x)$.  
If $\psi$ is uniformly prox-regular around $\overline{x}$ on $\mathbb{B}(\overline{x}; \eta)$ with $\eta>0$, we have $\partial_P\psi(x)=\partial_L\psi(x)$ for all $x\in\mathbb{B}(\overline{x}; \eta)$. In particular, $\partial_P\psi(x)=\partial_L\psi(x)$
 if $\psi$ is a semi-convex (convex) function.

Throughout the rest of  this paper, we make the following assumption on $f$ and $g$.
\begin{assumption}[H$_1$]\label{assump1}
\begin{itemize}
\item[{\rm(i)}] $f:\RR^n\rightarrow (-\infty,\infty]$ is a nonconvex differentiable function with $\mathbf{dom}~f$ convex and with its gradient $\nabla f$ being $L-$Lipschitz continuous on
$\mathbf{dom}~f$.
\item[{\rm(ii)}] $g$ is continuous on $\mathbf{dom}~g$, and $\mathbf{dom}~g$ is a convex set.
\item[{\rm(iii)}] $F$ is level-bounded i.e., the set $\{x\in\RR^n|F(x)\leq r\}$ is bounded (possibly empty) for every $r\in\RR$.
\end{itemize}
\end{assumption}
A few remarks about Assumption 1 are in order.
By Theorem 3.2.12 of \cite{Ortega}, the following descent property of $f$  holds
\begin{equation*}
\frac{L}{2}\|y-x\|^2+\langle\nabla f(x),y-x\rangle\geq f(y)-f(x)~~\hspace{4mm}\forall x,y\in \mathbf{dom}~f.
\end{equation*}
From {\rm{(i)} and {\rm{(ii)}}, $\mathbf{dom}~F$ is a convex set. In addition,
as a consequence of {\rm{(iii)},  the optimal value $F^*$ of (P) is finite and the optimal solution set $\mathbf{X}^*$ of (P) is non-empty.

For problem ($P$), by Exercise 2.3 of \cite{Clarke98}, if  $x\in\RR^n$ is a local minimizer of $F$, then
\begin{equation}\label{eq:pF}
0\in\partial_P F(x).
\end{equation}
A point satisfying~\eqref{eq:pF} is called a proximal critical point. The set of all proximal critical points of $F$ is denoted by $\overline{\mathbf{X}}_P$. By Assumption~\ref{assump1}, a global minimizer exists for problem (P). Hence, we have that $\overline{\mathbf{X}}_P\neq\emptyset$ since $\overline{\mathbf{X}}_P$ contains all global minimizers of (P). For the limiting subdifferential case, we can define the limiting critical point defined as:
$$\overline{\mathbf{X}}_L:=\{x|0\in\nabla f(x)+\partial_L g(x)\}.$$
Since $\partial_P g(x)\subseteq\partial_L g(x)$, we have $\overline{\mathbf{X}}_P\subseteq\overline{\mathbf{X}}_L$ and they coincide when $\partial_Pg(x)=\partial_Lg(x)$.

Recall (Definition~1.23\cite{Rockafellar}) that
$\psi: \RR^n \rightarrow \RR \cup \{\infty\}$
 is {\em prox-bounded} if there exists $\lambda>0$ such that
$$\inf_w\{\psi(w)+\frac{1}{2\lambda}||w-x||^2\}>-\infty\hspace{4mm}\mbox{for some $x\in \RR^n$}. $$ The supremum of the set of all such $\lambda$ is the threshold
$\lambda_{\psi}$ of prox-boundedness for $\psi$. A consequence of Assumption 1 is the prox-boundedness of $g$ and
a calculus rule for $\partial_P F$. We state them as the following proposition.
\begin{proposition}\label{prop1.1}
Suppose that Assumption 1 holds. Then

\noindent (i) $g$ is prox-bounded and $\lambda_g\geq \frac{1}{L}$ (with``$\frac{1}{0}=\infty$"). Moreover, if $\psi=g+\mbox{affine function of $x$}$, then $\lambda_{\psi}=\lambda_g$.

\noindent (ii) $\partial_P F(x)=\nabla f(x)+\partial_P g(x).$
\end{proposition}
\begin{proof} By Assumption~1, there is some $x_0$ such that
$$f(x)+g(x)\geq f(x_0)+g(x_0)~\hspace{10mm}\forall x.$$
As $\nabla f$ is Lipschitz continuous with constant $L$,
$$g(x)-g(x_0)\geq -(f(x)-f(x_0))\geq-\nabla f(x_0)^T(x-x_0)-\frac{L}{2}||x-x_0||^2，~~\forall x.$$
So
$$-\frac{r_g}{2}=\liminf_{|x|\rightarrow \infty}\frac{g(x)}{||x||^2}=\liminf_{|x|\rightarrow \infty}\frac{g(x)-g(x_0)}{||x||^2}\geq \liminf_{|x|\rightarrow \infty}\frac{-\nabla f(x_0)^T(x-x_0)}{||x||^2}-\frac{L}{2}\liminf_{|x|\rightarrow \infty}\frac{||x-x_0||^2}{||x||^2}=-\frac{L}{2}.$$
By Exercise 1.24 (d) of \cite{Rockafellar}, $g$ is prox-bounded. Again by the last part of Exercise 1.24 ,  a simple computation shows $\lambda_{g}=\frac{1}{\max\{0,r_g\}}\geq \frac{1}{L}$ ($\infty$ if $L=0$).
The assertion on $\psi$ follows again from Exercise 1.24 (d) as $-\frac{r_g}{2}=\liminf_{|x|\rightarrow \infty}\frac{g(x)}{||x||^2}=\liminf_{|x|\rightarrow \infty}\frac{\psi (x)}{||x||^2}$.

\noindent (ii) This follows from Proposition 2.3 of \cite{Ye18}.

\end{proof}
\subsection{Variable Bregman distance and VBPG method}\label{sec:Variable Bregman}
Let a sequence of functions $\{K^k,k\in\mathbb{N}\}$ and positive numbers $\{\epsilon^k,k\in\mathbb{N}\}$ be given, where the function $K^k$  is strongly convex and gradient Lipschitz. For each $k$, define a variable Bregman distance
\begin{equation}
D^k(x,y)=K^k(y)-[K^k(x)+\langle\nabla K^k(x),y-x\rangle],
\end{equation}
 The variable Bregman distance $D^k$ measures the proximity between two points $(x,y)$; that is,
$D^k(x,y)\geq0$,and  $D^k(x,y)=0$ if and only if  $x=y$.\\

Following \cite{Cohen80}, we propose to solve the partial nonsmooth and nonconvex problem of (P) by generating a sequence $\{x^k\}$  via the following
variable Bregman proximal gradient (VBPG) method :\\
\noindent\rule[0.25\baselineskip]{\textwidth}{1.5pt}
{\bf Variable Bregman Proximal Gradient method (VBPG)}\\
\noindent\rule[0.25\baselineskip]{\textwidth}{0.5pt}
{Initialize} $x^0\in\RR^n$\\
 \textbf{for} $k = 0,1,\cdots $, \textbf{do}
\begin{eqnarray}\label{APk}
\mbox{(AP$^k$)}\qquad x^{k+1}\in\arg\min_{x\in\RR^n}\bigg{\{}\langle\nabla f(x^k),x-x^k\rangle+g(x)+\frac{1}{\epsilon^k}D^k(x^k,x)\bigg{\}}.
\end{eqnarray}
\textbf{end for}\\
\noindent\rule[0.25\baselineskip]{\textwidth}{1.5pt}
Within each iteration of the VBPG method, the objective function of the minimization subproblem (AP$^k$) consists of two parts: the sum of linearized  $f$ at $x^k$ and $g$, and regularized term involving a variable proximal distance function part.\\
In the next section, we will provide conditions under which a solution of (AP$^k$) exists.
\subsection{Bregman type mappings and functions and their properties}\label{sec:Bregman}
The analysis of convergence and rate of convergence for the VBPG method,  relies crucially on Bregman type  mappings and functions. We now make the following standing assumption on the functions $K^k(x)$.
\begin{assumption}[H$_2$]\label{assump2}
\begin{itemize}
\item[{\rm(i)}] For each $k$, $K^k$ is strongly convex with $m^k$ and with its gradient $\nabla K^k$ being $M^k$-Lipschitz. In addition,
there are  $m>0$ and $M>0$ such that for all $k$, $m^k\geq m$, $M^k\leq M$.
\item[{\rm(ii)}] The parameter $\epsilon^k$ satisfies:
$$0<\underline{\epsilon}\leq\epsilon^k\leq\overline{\epsilon}.$$
\end{itemize}
\end{assumption}
Assumption~2, the family of Bregman distances $\{D^k~|~k\in N\}$ uniformly satisfies:
\begin{eqnarray*}
&&m\|x-y\|^2\leq m^k\|x-y\|^2\leq\langle\nabla_xD^k(x,y),x-y\rangle\leq M^k\|x-y\|^2\leq M\|x-y\|^2,\\
&&m\|x-y\|^2\leq m^k\|x-y\|^2\leq\langle\nabla_yD^k(x,y),y-x\rangle\leq M^k\|x-y\|^2\leq M\|x-y\|^2,\\
&&\|\nabla_yD^k(x,y)\|\leq M^k\|x-y\|\leq M\|x-y\|^2,\\
&&\frac{m}{2}\|x-y\|^2\leq\frac{m^k}{2}\|x-y\|^2\leq D^k(x,y)\leq\frac{M^k}{2}\|x-y\|^2\leq\frac{M}{2}\|x-y\|^2.
\end{eqnarray*}
 To simply our analysis, in what follows, we will drop the sub-index $k$. Thanks to Assumption~2, the results we
will establish
hold for all $k$. To this end,
let  a strongly twice differentiable convex function $K$  along with a positive $\epsilon\in
 (\underline{\epsilon}, \overline{\epsilon})$ be given.
Suppose a Bregman distance $D$  is constructed based on $K$ and $\epsilon$. We assume that
$D$ satisfies the following conditions:
\begin{eqnarray*}
&&m\|x-y\|^2\leq \langle\nabla_xD(x,y),x-y\rangle \leq M\|x-y\|^2,\\
&&m\|x-y\|^2 \leq\langle\nabla_yD (x,y),y-x\rangle \leq M\|x-y\|^2,\\
&&\|\nabla_yD (x,y)\| \leq M\|x-y\|^2,\\
&&\frac{m}{2}\|x-y\|^2 \leq D (x,y) \leq\frac{M}{2}\|x-y\|^2.
\end{eqnarray*}
Now we are ready to introduce
 the following mappings and functions which  will play a key role for the analysis of convergence and rate of convergence  for the VBPG method.\\
{\bf Bregman Proximal Envelope Function (BP Envelope Function)}\\
BP envelope function $E_{D,\epsilon}$ is defined by
\begin{equation}
E_{D,\epsilon}(x)=\min_{y\in\RR^n}\{f(x)+\langle\nabla f(x),y-x\rangle+g(y)+\frac{1}{\epsilon}D(x,y)\},\quad\forall x\in\RR^n,
\end{equation}
which is expressed as the value function of optimization problem (AP$^k$) (see (\ref{APk})), where $x^k$ is replaced by $x$.\\
{\bf Bregman Proximal Mapping}\\
Bregman proximal mapping $T_{D,\epsilon}$ is defined by
\begin{equation}\label{defi:Tk}
T_{D,\epsilon}(x)=\arg\min_{y\in\RR^n}\langle\nabla f(x),y-x\rangle+g(y)+\frac{1}{\epsilon}D(x,y),\quad\forall x\in\RR^n,
\end{equation}
which can be viewed as the set of  optimizers of optimization problem (AP$^k$), where $x^k$ is replaced by $x$.  Generally speaking, $T_{D,\epsilon}(x)$ could be multi-valued or even an empty set.\\
{\bf Bregman proximal gap function (BP gap function)}\\
Another useful nonnegative function $G_{D,\epsilon}$ (BP gap function) is defined by
\begin{eqnarray}
G_{D,\epsilon}(x)=-\frac{1}{\epsilon}\min_{y\in\RR^n}\{\langle\nabla f(x),y-x\rangle+g(y)-g(x)+\frac{1}{\epsilon}D(x,y)\},\quad\forall x\in\RR^n.
\end{eqnarray}
Obviously, we have $G_{D,\epsilon}(x)\geq0$ for all $x$. The following optimization problem is equivalent to the differential inclusion problem $0\in\partial_P F(x)$ associated with problem (P)
\begin{equation}
\min_{x\in\RR^n}G_{D,\epsilon}(x).
\end{equation}
The
above mappings and functions enjoy some favorable properties. These properties are summarized in the following propositions.
\begin{proposition}{\bf(Non-emptiness of values of $T_{D,\epsilon}$,~global properties of Bregman type mappings and functions)}\label{prop:Ek} Let a  Bregman function $D$ be given.
Suppose that Assumptions~\ref{assump1}  and \ref{assump2} hold, and that $\epsilon\in (0,m/L)$. Then for any $x\in\RR^n$, {\rm(i)} $T_{D,\epsilon}(x)$ is nonempty and compact;
moreover, for any given $t_{D,\epsilon}(x)\in T_{D,\epsilon}(x)$, we have
{\rm
\begin{itemize}
\item[(ii)] $E_{D,\epsilon}(x)=F(x)-\epsilon G_{D,\epsilon}(x)$;
\item[(iii)] $F\big{(}t_{D,\epsilon}(x)\big{)}\leq E_{D,\epsilon}(x)-\frac{1}{2}(\frac{m}{\overline{\epsilon}}-L)\|x-t_{D,\epsilon}(x)\|^2$;
\item[(iv)] $F\big{(}t_{D,\epsilon}(x)\big{)}\leq F(x)-\frac{1}{2}\big{(}\frac{m}{\overline{\epsilon}}-L\big{)}\|x-t_{D,\epsilon}(x)\|^2$.
\end{itemize}
}
\end{proposition}
\begin{proof}
{\rm (i):} For any given $x\in dom~(F)$, since $0<\epsilon<m/L$, by Proposition \ref{prop1.1},
$$\{y~|~\langle \nabla f(x),y-x\rangle+g(y)+\frac{m}{2\epsilon}||x-y||^2\leq \alpha\}$$
is level-bounded for any $\alpha\in \RR$. For any $y$, by assumption, we have
$$\langle \nabla f(x),y-x\rangle+g(y)+\frac{D(x,y)}{\epsilon} \geq
\langle f(x),y-x\rangle +g(y)+\frac{m}{2\overline{\epsilon}} ||x-y||^2.$$
So the set $$\{y~|~\langle \nabla f(x),y-x\rangle+g(y)+\frac{1}{\epsilon}D(x,y)\leq \alpha\}
\subset   \{y~|~\langle \nabla f(x),y-x\rangle+g(y)+\frac{m}{2\overline{\epsilon}}||x-y||^2\leq \alpha\}$$
is level-bounded for any $\alpha\in \RR$. By Theorem~1.9 of \cite{Rockafellar}, $T_{D,\epsilon}(x)$ is non-empty and
compact.\\
{\rm (ii):} This follows   immediately from the definitions $G_{D,\epsilon}(x)$ and $E_{D,\epsilon}(x)$.\\
{\rm (iii) \& (iv):} Since $\nabla f$ is $L$-Lipschitz
\begin{eqnarray}
E_{D,\epsilon}(x)&=&f(x)+\langle\nabla f(x),t_{D,\epsilon}(x)-x\rangle+g\big{(}t_{D,\epsilon}(x)\big{)}+\frac{1}{\epsilon}D\big{(}x,t_{D,\epsilon}(x)\big{)}\nonumber\\
      &\geq&f\big{(}t_{D,\epsilon}(x)\big{)}-\frac{L}{2}\|x-t_{D,\epsilon}(x)\|^2+g\big{(}t_{D,\epsilon}(x)\big{)}+\frac{1}{\epsilon}D\big{(}x,t_{D,\epsilon}(x)\big{)}
\end{eqnarray}
Thus
\begin{eqnarray}
F\big{(}t_{D,\epsilon}(x)\big{)}&\leq& E_{D,\epsilon}(x)-\frac{1}{\epsilon}D\big{(}x,t_{D,\epsilon}(x)\big{)}+\frac{L}{2}\|x-t_{D,\epsilon}(x)\|^2\nonumber\\
                     &\leq& E_{D,\epsilon}(x)-\frac{1}{2}(\frac{m}{\overline{\epsilon}}-L)\|x-t_{D,\epsilon}(x)\|^2\nonumber\\
                     &&\qquad\mbox{(since $D\big{(}x,t_{D,\epsilon}(x)\big{)}\geq\frac{m}{2}\|x-t_{D,\epsilon}(x)\|^2$ and $\epsilon\leq\overline{\epsilon}$)}\nonumber\\
                     &\leq& F(x)-\frac{1}{2}(\frac{m}{\overline{\epsilon}}-L)\|x-t_{D,\epsilon}(x)\|^2 \hspace{5mm}(by (ii)).
\end{eqnarray}
\end{proof}
\begin{proposition}\label{prop:Fk}{\bf (Properties of $\partial_P F$)}
Suppose that Assumptions~\ref{assump1} and \ref{assump2} hold. Then for all $t_{D,\epsilon}(x)\in T_{D,\epsilon}(x)$ we have
\begin{itemize}
\item[{\rm(i)}] $\xi=\nabla f\big{(}t_{D,\epsilon}(x)\big{)}-\nabla f(x)-\frac{1}{\epsilon}\nabla_y D\big{(}x,t_{D,\epsilon}(x)\big{)}\in\partial_P F\big{(}t_{D,\epsilon}(x)\big{)}$;
\item[{\rm(ii)}] $dist\bigg{(}0,\partial_PF\big{(}t_{D,\epsilon}(x)\big{)}\bigg{)}\leq(L+\frac{M}{\underline{\epsilon}})\|x-t_{D,\epsilon}(x)\|$;
\item[{\rm(iii)}] If $x\in T_{D,\epsilon}(x)$, then $0\in\partial_P F(x)$.
\end{itemize}
\end{proposition}
\begin{proof}
{\rm(i):} Writing down the optimality condition of optimizer $t_{D,\epsilon}(x)\in T_{D,\epsilon}(x)$ yields
\begin{equation}
0\in\nabla f(x)+\partial_P g\big{(}t_{D,\epsilon}(x)\big{)}+\frac{1}{\epsilon}\nabla_yD\big{(}x,t_{D,\epsilon}(x)\big{)}~~(\mbox{by Proposition~\ref{prop1.1}~(ii)})
\end{equation}

Let $\xi=\nabla f\big{(}t_{D,\epsilon}(x)\big{)}-\nabla f(x)-\frac{1}{\epsilon}\nabla_yD\big{(}x,t_{D,\epsilon}(x)\big{)}$. Then we have
\begin{equation}
\xi\in\partial_P F\big{(}t_{D,\epsilon}(x)\big{)}=\nabla f(t_{D,\epsilon}(x))+\partial_P g(t_{D,\epsilon}(x)).
\end{equation}
{\rm(ii):} By the expression of $\xi$ in (i) and Assumption~\ref{assump2}, we have
\begin{eqnarray}
\|\xi\|&\leq&\|\nabla f\big{(}t_{D,\epsilon}(x)\big{)}-\nabla f(x)\|+\frac{1}{\underline{\epsilon}}\|\nabla_yD\big{(}x,t_{D,\epsilon}(x)\big{)}\|\nonumber\\
&\leq&(L+\frac{M}{\underline{\epsilon}})\|x-t_{D,\epsilon}(x)\|,
\end{eqnarray}
which follows the desired statement.\\
{\rm(iii):} The claim follows directly from statements (i) and (ii).
\end{proof}
\begin{proposition}[Continuity for $E_{D,\epsilon}(x)$, $G_{D,\epsilon}(x)$ and $T_{D,\epsilon}(x)$]\label{prop:1.4}
Suppose assumptions of Proposition~\ref{prop:Ek} hold. If $\overline{\epsilon}<\frac{m}{L}$, then function $E_{D,\epsilon}(x)$ and $G_{D,\epsilon}(x)$ are continuous, mapping $T_{D,\epsilon}(x)$ is closed and is continuous whenever $T_{D,\epsilon}(x)$ is single valued.
\end{proposition}
\begin{proof}
Let $\varphi(x,y)=f(x)+\langle\nabla f(x),y-x\rangle+\frac{1}{\epsilon}D(x,y)$, then we have $E_{D,\epsilon}(x)=\min\limits_{y\in\RR^n}\varphi(x,y)+g(y)$ and $T_{D,\epsilon}(x)=\arg\min\limits_{y\in\RR^n}\varphi(x,y)+g(y)$.\\
First, we show that $E_{D,\epsilon}(x)$ is u.s.c. Let $x_n\rightarrow\overline{x}$. For $\overline{x}$, there is $\overline{y}\in T_{D,\epsilon}(\overline{x})$ such that $E_{D,\epsilon}(\overline{x})=\varphi\big{(}\overline{x},\overline{y}\big{)}+g\big{(}\overline{y}\big{)}$.\\
Since $E_{D,\epsilon}(x_n)=\min\limits_{y\in\RR^n}\varphi(x_n,y)+g(y)\leq\varphi(x_n,\overline{y})+g(\overline{y})$, then
\begin{eqnarray}
\lim_{x_n\rightarrow\overline{x}}\sup E_{D,\epsilon}(x_n)\leq\varphi(\overline{x},\overline{y})+g(\overline{y})=E_{D,\epsilon}(\overline{x}),\qquad\mbox{(by the continuity of $\varphi(\cdot,\overline{y})$)}
\end{eqnarray}
which shows $E_{D,\epsilon}(x)$ is u.s.c.\\
Another hand, from Maximum theorem (Theorem 1, P115) in book of Berge~\cite{Berge97}, $E_{D,\epsilon}(x)$ is l.s.c, therefore the continuity of $E_{D,\epsilon}(x)$ is provided. Next, we will show the continuity of $T_{D,\epsilon}(x)$. For $x_n\rightarrow\overline{x}$, any $y_n\in T_{D,\epsilon}(x_n)$, from (iii) of Proposition~\ref{prop:Ek}, we have
\begin{eqnarray}
\frac{1}{2}\big{(}\frac{m}{\overline{\epsilon}}-L\big{)}\|x_n-y_n\|^2&\leq& E_{D,\epsilon}(x_n)-F\big{(}y_n\big{)}\nonumber\\
                             &\leq& E_{D,\epsilon}(x_n)-F^*,
\end{eqnarray}
which shows that $T_{D,\epsilon}(x_n)$ is bounded when $x_n\rightarrow\overline{x}$. Taking the subsequence $\{x_{n'}\}\subset\{x_n\}$ such that $y_{n'}\in T_{D,\epsilon}(x_{n'})$, and $y_{n'}\rightarrow\hat{y}$.\\
By the continuity of $E_{D,\epsilon}(x)$, we have
\begin{eqnarray}
E_{D,\epsilon}(\overline{x})=\lim_{x_{n'}\rightarrow\overline{x}}E_{D,\epsilon}(x_{n'})&=&\lim_{x_{n'}\rightarrow \overline{x}}\varphi(x_{n'},y_{n'})+g(y_{n'})\nonumber\\
&\geq&\varphi(\overline{x},\hat{y})+g(\hat{y}).\\
&&\mbox{(since $g(\cdot)$ is l.s.c and $\varphi(\cdot,\cdot)$ is continuous)}\nonumber
\end{eqnarray}
which shows that $\hat{y}\in T_{D,\epsilon}(\overline{x})$ and $T_{D,\epsilon}(x)$ is closed and continuous whenever $T_{D,\epsilon}(x)$ is unique valued.\\
From the definition of BP envelope function and BP gap function, we have
$$G_{D,\epsilon}(x)=\frac{1}{\epsilon}\left[F(x)-E_{D,\epsilon}(x)\right].$$
Therefore, $G_{D,\epsilon}(x)$ is also continuous.
\end{proof}
Before the end of this section, we introduce the following lemma about the generalized descent inequality. We use $\mathcal{F}(\RR^n)$ to denote the set of $\mathcal{C}^1$-smooth functions from $\RR^n$ to $(-\infty,+\infty]$; $\Gamma(\RR^n)$ the set of proper and lower semicontinuous functions from $\RR^n$ to $(-\infty,+\infty]$.
\begin{lemma}[Generalized descent inequality in the nonconvex case]\label{lemma:1}
Suppose that Assumptions~\ref{assump1} and~\ref{assump2} hold. Then for any $t_{D,\epsilon}(x)\in T_{D,\epsilon}(x)$, $x\in\RR^n$, we have that
\begin{eqnarray}\label{eq:descent}
\mathfrak{a}\left[F(t_{D,\epsilon}(x))-F(u)\right]\leq \mathfrak{b}\|u-x\|^2-\|u-t_{D,\epsilon}(x)\|^2-\mathfrak{c}\|x-t_{D,\epsilon}(x)\|^2,~\forall u\in\RR^n
\end{eqnarray}
where the values of $\mathfrak{a}$, $\mathfrak{b}$ and $\mathfrak{c}$ are given as follows:
\begin{itemize}
\item[{\rm(i)}] $\mathfrak{a}=2$, $\mathfrak{b}=\frac{M}{\underline{\epsilon}}+2+3L$ and $\mathfrak{c}=\frac{m}{\overline{\epsilon}}-(L+2)$, when $f\in\mathcal{F}(\RR^n)$ and $g\in\Gamma(\RR^n)$;
\item[{\rm(ii)}] $\mathfrak{a}=2$, $\mathfrak{b}=\frac{M}{\underline{\epsilon}}+2$ and $\mathfrak{c}=\frac{m}{\overline{\epsilon}}-(L+2)$, when $f\in\mathcal{F}(\RR^n)$ convex and $g\in\Gamma(\RR^n)$;
\item[{\rm(iii)}] $\mathfrak{a}=\frac{2\overline{\epsilon}}{m}$, $\mathfrak{b}=\frac{M}{m}+\frac{3L\overline{\epsilon}}{m}$ and $\mathfrak{c}=1-\frac{L\overline{\epsilon}}{m}$, when $f\in\mathcal{F}(\RR^n)$ and $g\in\Gamma(\RR^n)$ convex;
\item[{\rm(iv)}] $\mathfrak{a}=\frac{2\overline{\epsilon}}{m}$, $\mathfrak{b}=\frac{M}{m}$ and $\mathfrak{c}=1-\frac{L\overline{\epsilon}}{m}$, when $f\in\mathcal{F}(\RR^n)$ convex and $g\in\Gamma(\RR^n)$ convex.
\end{itemize}
\end{lemma}
\begin{proof}
Denote $\Delta=\langle\nabla f(x),t_{D,\epsilon}(x)-u\rangle+g(t_{D,\epsilon}(x))-g(u)$. First, we estimate the lower bound of $\Delta$:\\
\begin{eqnarray}\label{eq:bound1-f-n}
\Delta&=&\langle\nabla f(x),t_{D,\epsilon}(x)-u\rangle+g(t_{D,\epsilon}(x))-g(u)\nonumber\\
&=&\langle\nabla f(x),t_{D,\epsilon}(x)-x\rangle+\langle\nabla f(x),x-u\rangle+g(t_{D,\epsilon}(x))-g(u)\nonumber\\
&\geq&f\left(t_{D,\epsilon}(x)\right)-f(x)-\frac{L}{2}\|x-t_{D,\epsilon}(x)\|^2+\langle\nabla f(x),x-u\rangle+g(t_{D,\epsilon}(x))-g(u)\nonumber\\
&&\qquad\qquad\qquad\qquad\qquad\qquad\mbox{(since $f$ is gradient Lipschitz with exponent $L$)}\nonumber\\
&=&F\left(t_{D,\epsilon}(x)\right)-F(u)-\frac{L}{2}\|x-t_{D,\epsilon}(x)\|^2+\underbrace{f(u)-f(x)-\langle\nabla f(x),u-x\rangle}_{\delta_1}
\end{eqnarray}
Next we estimate the term $\delta_1$ in~\eqref{eq:bound1-f-n}. If $f\in\mathcal{F}(\RR^n)$ convex, we have $\delta_1\geq0$. For the case $f\in\mathcal{F}(\RR^n)$ without convexity, we have that
\begin{eqnarray}
\delta_1&=&f(u)-f(x)-\langle\nabla f(x),u-x\rangle\nonumber\\
&=&f(u)-f(x)-\langle\nabla f(u),u-x\rangle+\langle\nabla f(x)-\nabla f(u),x-u\rangle\nonumber\\
&\geq&-\frac{L}{2}\|u-x\|^2-\|\nabla f(x)-\nabla f(u)\|\cdot\|x-u\|\nonumber\\
&&\qquad\qquad\qquad\qquad\qquad\qquad\mbox{(since $f$ is gradient Lipschitz with $L$)}\nonumber\\
&\geq&-\frac{3L}{2}\|u-x\|^2.\qquad\mbox{(since $f$ is gradient Lipschitz with $L$)}
\end{eqnarray}
Therefore, we have that
\begin{eqnarray}\label{eq:bound1}
\Delta\geq\left\{\begin{array}{ll}
F\left(t_{D,\epsilon}(x)\right)-F(u)-\frac{L}{2}\|x-t_{D,\epsilon}(x)\|^2                      &\mbox{if}~f\in\mathcal{F}(\RR^n)~\mbox{convex}\\
F\left(t_{D,\epsilon}(x)\right)-F(u)-\frac{L}{2}\|x-t_{D,\epsilon}(x)\|^2-\frac{3L}{2}\|u-x\|^2&\mbox{if}~f\in\mathcal{F}(\RR^n)
\end{array}\right.
\end{eqnarray}
Now turn to estimate the upper bound of $\Delta$. If $g\in\Gamma(\RR^n)$ and convex, the optimal condition of the minimization problem in Bregman proximal mapping~\eqref{defi:Tk} is given by the following variational inequality
\begin{eqnarray}\label{eq:optimalcondition-gc0}
\langle\nabla f(x),t_{D,\epsilon}(x)-u\rangle+g(t_{D,\epsilon}(x))-g(u)+\frac{1}{\epsilon}\langle\nabla K\left(t_{D,\epsilon}(x)\right)-\nabla K(x), t_{D,\epsilon}(x)-u\rangle\leq0,
\end{eqnarray}
or
\begin{eqnarray}\label{eq:optimalcondition-gc1}
\Delta&=&\langle\nabla f(x),t_{D,\epsilon}(x)-u\rangle+g(t_{D,\epsilon}(x))-g(u)\nonumber\\
&\leq&\frac{1}{\epsilon}\langle\nabla K\left(t_{D,\epsilon}(x)\right)-\nabla K(x), u-t_{D,\epsilon}(x)\rangle\nonumber\\
&=&\frac{1}{\epsilon}\left[D(x,u)-D(t_{D,\epsilon}(x),u)-D(x,t_{D,\epsilon}(x))\right]\nonumber\\
&\leq&\frac{M}{2\underline{\epsilon}}\|u-x\|^2-\frac{m}{2\overline{\epsilon}}\|u-t_{D,\epsilon}(x)\|^2-\frac{m}{2\overline{\epsilon}}\|x-t_{D,\epsilon}(x)\|^2.\qquad\mbox{(by Assumption~\ref{assump2})}
\end{eqnarray}
For the case $g\in\Gamma(\RR^n)$ without convexity, then the optimal condition of the minimization problem in Bregman proximal mapping~\eqref{defi:Tk} is given by
\begin{eqnarray}\label{eq:optimalcondition-gn0}
\langle\nabla f(x),t_{D,\epsilon}(x)-u\rangle+g(t_{D,\epsilon}(x))-g(u)+\frac{1}{\epsilon}\left[D(x,t_{D,\epsilon}(x))-D(x,u)\right]\leq0,
\end{eqnarray}
or
\begin{eqnarray}\label{eq:optimalcondition-gn1}
\Delta&=&\langle\nabla f(x),t_{D,\epsilon}(x)-u\rangle+g(t_{D,\epsilon}(x))-g(u)\nonumber\\
&\leq&\frac{1}{\epsilon}\left[D(x,u)-D(x,t_{D,\epsilon}(x))\right]\nonumber\\
&\leq&\frac{M}{2\underline{\epsilon}}\|u-x\|^2-\frac{m}{2\overline{\epsilon}}\|x-t_{D,\epsilon}(x)\|^2\qquad\qquad\qquad\qquad\qquad\qquad\mbox{(by Assumption~\ref{assump2})}\nonumber\\
&\leq&\frac{M}{2\underline{\epsilon}}\|u-x\|^2-\frac{1}{2}\|u-t_{D,\epsilon}(x)\|^2+\frac{1}{2}\|u-t_{D,\epsilon}(x)\|^2-\frac{m}{2\overline{\epsilon}}\|x-t_{D,\epsilon}(x)\|^2\nonumber\\
&\leq&\frac{M}{2\underline{\epsilon}}\|u-x\|^2-\frac{1}{2}\|u-t_{D,\epsilon}(x)\|^2+\|u-x\|^2+\|x-t_{D,\epsilon}(x)\|^2-\frac{m}{2\overline{\epsilon}}\|x-t_{D,\epsilon}(x)\|^2\nonumber\\
&\leq&\left(\frac{M}{2\underline{\epsilon}}+1\right)\|u-x\|^2-\frac{1}{2}\|u-t_{D,\epsilon}(x)\|^2-\frac{m-2\overline{\epsilon}}{2\overline{\epsilon}}\|x-t_{D,\epsilon}(x)\|^2.
\end{eqnarray}
Together~\eqref{eq:optimalcondition-gc1} and~\eqref{eq:optimalcondition-gn1}, we have that
\begin{eqnarray}\label{eq:bound2}
\Delta\leq\left\{\begin{array}{ll}
\frac{M}{2\underline{\epsilon}}\|u-x\|^2-\frac{m}{2\overline{\epsilon}}\|u-t_{D,\epsilon}(x)\|^2-\frac{m}{2\overline{\epsilon}}\|x-t_{D,\epsilon}(x)\|^2                      &\mbox{if}~g\in\Gamma(\RR^n)~\mbox{convex}\\
\left(\frac{M}{2\underline{\epsilon}}+1\right)\|u-x\|^2-\frac{1}{2}\|u-t_{D,\epsilon}(x)\|^2-\frac{m-2\overline{\epsilon}}{2\overline{\epsilon}}\|x-t_{D,\epsilon}(x)\|^2&\mbox{if}~g\in\Gamma(\RR^n)
\end{array}\right.
\end{eqnarray}
Combing~\eqref{eq:bound1} and~\eqref{eq:bound2} we can construct the following descent inequality
\begin{eqnarray}\label{eq:descent}
\mathfrak{a}\left[F(t_{D,\epsilon}(x))-F(u)\right]+\|u-t_{D,\epsilon}(x)\|^2\leq \mathfrak{b}\|u-x\|^2-\mathfrak{c}\|x-t_{D,\epsilon}(x)\|^2,
\end{eqnarray}
where parameters $\mathfrak{a}$, $\mathfrak{b}$ and $\mathfrak{c}$ are shown as follows.
\begin{table}[htbp]
	\centering  
	\caption{Parameters in descent inequality~\eqref{eq:descent}}  
	\label{table1}  
	\begin{tabular}{|c|l|c|c|c|}
        \hline
        {\bf NO.}&{\bf Problem}&$\mathfrak{a}$&$\mathfrak{b}$&$\mathfrak{c}$\\
		\hline  
		1.&$f\in\mathcal{F}(\RR^n)$ and $g\in\Gamma(\RR^n)$&$2$&$\frac{M}{\underline{\epsilon}}+2+3L$&$\frac{m}{\overline{\epsilon}}-(L+2)$ \\  
		\hline
		2.&$f\in\mathcal{F}(\RR^n)$ convex and $g\in\Gamma(\RR^n)$&$2$&$\frac{M}{\underline{\epsilon}}+2$&$\frac{m}{\overline{\epsilon}}-(L+2)$ \\
		\hline
        3.&$f\in\mathcal{F}(\RR^n)$ and $g\in\Gamma(\RR^n)$ convex&$\frac{2\overline{\epsilon}}{m}$&$\frac{M}{m}+\frac{3L\overline{\epsilon}}{m}$&$1-\frac{L\overline{\epsilon}}{m}$ \\
		\hline
        4.&$f\in\mathcal{F}(\RR^n)$ convex and $g\in\Gamma(\RR^n)$ convex&$\frac{2\overline{\epsilon}}{m}$&$\frac{M}{m}$&$1-\frac{L\overline{\epsilon}}{m}$ \\
		\hline
	\end{tabular}
\end{table}
\end{proof}
\begin{remark}
By Lemma~\ref{lemma:1}, we can derive some useful results. For example, for the critical point $\overline{x}$, taking $u=\overline{x}$, if $F\left(t_{D,\epsilon}(x)\right)\geq F(\overline{x})$, $x\in\RR^n$ from Lemma~\ref{lemma:1}, then we obtain $\|t_{D,\epsilon}(x)-t_{D,\epsilon}(\overline{x})\|\leq\sqrt{\frac{b}{a}}\|x-\overline{x}\|$, i.e. mapping $T_{D,\epsilon}(x)$ is Lipschitz around $\overline{x}$:
$$T_{D,\epsilon}(x)\subset T_{D,\epsilon}(\overline{x})+\sqrt{\frac{b}{a}}\|x-\overline{x}\|\cdot\mathbb{B}.$$
From this lemma, we also get for $x,u\in\RR^n$
\begin{eqnarray}
F\left(t_{D,\epsilon}(x)\right)-F(u)&\leq&\frac{1}{\mathfrak{a}}\left\{2\mathfrak{b}\|u-t_{D,\epsilon}(x)\|^2+2\mathfrak{b}\|t_{D,\epsilon}(x)-x\|^2-\|u-t_{D,\epsilon}(x)\|^2-\mathfrak{c}\|x-t_{D,\epsilon}(x)\|^2\right\}\nonumber\\
&\leq&\kappa\left(\|u-t_{D,\epsilon}(x)\|^2+\|x-t_{D,\epsilon}(x)\|^2\right),\quad\kappa=\max\{\frac{\mathfrak{b}-1}{\mathfrak{a}},\frac{\mathfrak{b}-\mathfrak{c}}{\mathfrak{a}}\},
\end{eqnarray}
which is one cost-to-go estimate used in~\cite{Tseng2009}.
\end{remark}

\section{Quantitative analysis of level sets and level-set based error bounds}\label{sec:level-set}
\subsection{Level-set analysis}
Given an $\overline{x}\in \mathbf{dom}~F$, let $\overline{F}=F(\bar x)$.  Set  $[F\leq\overline{F}]=\{x\in\RR^n|F(x)\leq F(\overline{x})\}$ and $[F>\overline{F}]=\{x\in\RR^n|F(x)>F(\overline{x})\}$. In this subsection, we present two level-set results for $F$ which will be useful in the following sections.
\begin{lemma}\label{lemma1.5}
Let  $\overline{x}\in\RR^n$ be given.  For any $x\in[F>\overline{F}]$, the function value of $F$at  the projection of $x$ on the level set $[F\leq\overline{F}]$ is $\overline{F}$; that is,
$$x_p\in{Proj}_{[F\leq\overline{F}]}(x)\quad\mbox{and}\quad F(x_p)=\overline{F}.$$
\end{lemma}
\begin{proof}
If $F(x_p)<\overline{F}$, then  we set $\varphi(t)=F\big{(}tx_p+(1-t)x\big{)}$, $t\in[0,1]$.
Since $F$ is continuous on $\mathbf{dom}~F$ and $\mathbf{dom}~F$ is convex, $\varphi$ is continuous on $[0,1]$. So, there is some $t_0\in(0,1)$ such that $F(x_0)=\overline{F}$ with $x_0=t_0x_p+(1-t_0)x$. Then $\|x-x_0\|<\|x-x_p\|$, a  contradiction.
Hence $F(x_p)$ must be $\overline{F}$, and the proof is completed.
\end{proof}
With the help of the above lemma,  the next proposition provides the value proximity in terms of the distance between $x$ and the set $[F\leq\overline{F}]$. This proposition will play a pivotal role in the rate of convergence analysis of
the VBPG method.
\begin{proposition}[Function-value proximity in terms of  level sets]\label{proposition1.5}
Suppose that Assumptions~\ref{assump1} and \ref{assump2} hold. If $\overline{\epsilon}<\frac{m}{L}$, then there is some $c_0=\frac{3}{2}L+\frac{M}{2\underline{\epsilon}}>0$ such that the   following estimation holds.
$$F\left(t_{D,\epsilon}(x)\right)-\overline{F}\leq E_{D,\epsilon}(x)-\overline{F}\leq c_0{dist}^2(x,[F\leq\overline{F}]),\quad\forall x\in [F>\overline{F}],\hspace{3mm}\forall t_{D,\epsilon}(x)\in T_{D,\epsilon}(x).$$
\end{proposition}
\begin{proof} With the given choice of $\epsilon$,   $T_{D,\epsilon}(x)$ is nonempty by Proposition~\ref{prop:Ek}. So $E_{D,\epsilon}(x)$ has a finite value for any given $x$.
For $x\in [F>\overline{F}]$, let $x_p\in[F\leq\overline{F}]$ such that $\|x-x_p\|=dist(x,[F\leq\overline{F}])$. By Lemma~\ref{lemma1.5}, we have $F(x_p)=F(\overline{x})=\overline{F}$.
Now we estimate the difference $E_{D,\epsilon}(x)-\overline{F}$. As $\overline{\epsilon}<\frac{m}{L}$,
   by  (iii) of Proposition~\ref{prop:Ek}, we have
    \begin{eqnarray}\label{eq:4.1}
F(t_{D,\epsilon}(x))-\overline{F}&\leq & E_{D,\epsilon}(x)-\overline{F}   \nonumber        \\
&=&\min_{y\in\RR^n}\big{\{}f(x)+\langle\nabla f(x),y-x\rangle+g(y)+\frac{1}{\epsilon}D(x,y)\big{\}}-(f+g)(x_p)\nonumber\\
&\leq&f(x)+\langle\nabla f(x),x_p-x\rangle+g(x_p)+\frac{1}{\epsilon}D(x,x_p)-(f+g)(x_p)\nonumber\\
&=&f(x)-f(x_p)+\langle\nabla f(x),x_p-x\rangle+\frac{1}{\epsilon}D(x,x_p)\nonumber\\
&\leq&\langle\nabla f(x_p),x-x_p\rangle+\frac{L}{2}\|x-x_p\|^2+\langle\nabla f(x),x_p-x\rangle+\frac{1}{\epsilon}D(x,x_p)\quad\mbox{(by Assumption~\ref{assump1})}\nonumber\\
&=&\langle\nabla f(x_p)-\nabla f(x),x-x_p\rangle+\frac{L}{2}\|x-x_p\|^2+\frac{1}{\epsilon}D(x,x_p)\nonumber\\
&\leq&\frac{3}{2}L\|x-x_p\|^2+\frac{M}{2\underline{\epsilon}}\|x-x_p\|^2\qquad\qquad\qquad\qquad\qquad\qquad\qquad\qquad\;\mbox{(by Assumption~\ref{assump1})}\nonumber\\
&\leq&c_0\|x-x_p\|^2
=c_0dist^2(x,[F\leq\overline{F}])\qquad\mbox{(where $c_0=\frac{3}{2}L+\frac{M}{2\underline{\epsilon}}$).}
\end{eqnarray}
\end{proof}
\subsection{Level-set based error bounds}\label{subsec:level-set error bounds}
In this subsection, first we will introduce the concepts of level-set subdifferential  and level-set Bregman proximal error bounds, then we will discuss their relationships.
For given positive numbers $\eta$ and $\mu$, let
$$\mathfrak{B}(\overline{x};\eta,\nu)=\mathbb{B}(\overline{x};\eta)\cap\{x\in \RR^n~|~\overline{F}<F(x)<\overline{F}+\nu\}.$$
\begin{definition}[Level-set subdifferential error bound]\label{LSEB} The proper lower semicontinuous function $F$  is said to satisfy the level-set subdifferential error bound condition at $\overline{x}$ with exponent $\gamma>0$ if there exist $\eta>0$, $\nu>0$, and $c_3>0$ such that the following inequality holds:
$$dist^{\gamma}(x,[F\leq\overline{F}])\leq c_3dist\big{(}0,\partial_P F(x)\big{)}~~\forall
x\in\mathfrak{B}(\overline{x};\eta,\nu).$$
\end{definition}
\begin{definition}[Level-set Bregman proximal error bound]\label{LBEB}
Given a Bregman function $D$ along with $\epsilon>0$, we say that the function $F$ satisfies the level-set Bregman proximal (BP) error bound condition at $\overline{x}$ with exponent $p>0$ , if there exist $\eta>0$, $\nu>0$, and
$\theta>0$ such that the following inequality holds:
$$dist^p(x,[F\leq\overline{F}])\leq\theta dist\left(x,T_{D,\epsilon}(x)\right)~~\forall x\in\mathfrak{B}(\overline{x};\eta,\nu).$$
\end{definition}
{\bf Property (A)} Let a real number $\overline{F}$ be given.  We say that  $x\in\RR^n$ satisfies {\rm Property (A)} if $F\left(t_{D,\epsilon}(x)\right)\geq \overline{F}$ for all $t_{D,\epsilon}(x)\in T_{D,\epsilon}(x)$.

In the rest of this paper, unless otherwise stated, we will always choose $\overline{F}=F(\overline{x})$ for some
given $\overline{x}$. The following lemma reveals a consequence of Property (A): if $x$ is near a ball centered at $\overline{x}$  and $x$ satisfies Property (A), then all $t_{D,\epsilon}(x)$ are still in the ball.
\begin{lemma}\label{lemma:tp}
Let Bregman distance $D$,$\epsilon$, and $\overline{x}$ be given.
 Suppose that Assumptions~\ref{assump1} and~\ref{assump2} hold with $\overline{\epsilon}<\frac{m}{L}$. Let $x\in\mathfrak{B}(\overline{x};\frac{\eta}{2},\frac{\nu}{N})$ where $N\geq\frac{2\overline{\epsilon}\nu}{m-\overline{\epsilon}L}/\left(\frac{\eta}{2}\right)^2$ be given. If $x$ satisfies {\rm Property (A)}, then we have $\|x-t_{D,\epsilon}(x)\|\leq\frac{\eta}{2}$ and $t_{D,\epsilon}(x)\in\mathbb{B}(\overline{x};\eta)\cap\{x\in \RR^n~|~\overline{F}\leq F(x)<\overline{F}+\nu\}$ for all $t_{D,\epsilon}(x)\in T_{D,\epsilon}(x)$ .
\end{lemma}
\begin{proof}
As $x$ satisfies {\rm Property (A)}, $F\left(t_{D,\epsilon}(x)\right)\geq\overline{F}$ for all $t_{D,\epsilon}(x)\in T_{D,\epsilon}(x)$.
By (iv) of Proposition~\ref{prop:Ek}, we have that
$$\frac{1}{2}\left(\frac{m}{\overline{\epsilon}}-L\right)\|x-t_{D,\epsilon}(x)\|^2\leq F(x)-F\big{(}t_{D,\epsilon}(x)\big{)}\leq F(x)-\overline{F}\leq\frac{\nu}{N}.$$
Since $N\geq\frac{2\overline{\epsilon}\nu}{m-\overline{\epsilon}L}/(\frac{\eta}{2})^2$,  $\sqrt{\frac{2\nu\overline{\epsilon}}{N(m-\overline{\epsilon}L)}}\leq\frac{\eta}{2}$,
$$\|x-t_{D,\epsilon}(x)\|\leq\frac{\eta}{2}.$$
As $\| x-\overline{x}\|\leq\frac{\eta}{2}$,
it follows that $\|t_{D,\epsilon}(x)-\overline{x}\|\leq\|t_{D,\epsilon}(x)-x\|+\|x-\overline{x}\|\leq\eta$. This yields $t_{D,\epsilon}(x)\in\mathbb{B}(\overline{x};\eta)$.
\end{proof}
\begin{remark}
A few remarks on {\rm Property (A)} are in order. If $\overline{F}$ is the minimal value of $F$, then Property
(A) holds trivially. In  Sections~\ref{sec:convergence} and~\ref{sec:convergence rate}, when we consider a sequence $\{x^k\}_{k=0}^{\infty}$ generated by VBPG converging to a critical point $\overline{x}$,Property {\rm(A)} holds for $x\in\{x^k~|~k=1,\dots,\infty\}$.
\end{remark}
\begin{theorem}[Level-set subdifferential EB implies level-set Bregman proximal EB]\label{theo:4.2}
Suppose Assumptions~\ref{assump1} and~\ref{assump2} hold with $\overline{\epsilon}<\frac{m}{L}$. Assume the level-set subdifferential error bound holds at $\overline{x}$ with exponent $\gamma\in(0,\infty)$ over $\mathfrak{B}(\overline{x};\eta,\nu)$.
Then there are $N>\frac{2\overline{\epsilon}\nu}{m-\overline{\epsilon}L}/(\frac{\eta}{2})^2$, and $\theta>0$ such that
\begin{equation}
dist^p(x,[F\leq\overline{F}])\leq\theta dist\left(x,T_{D,\epsilon}(x)\right)\quad\mbox{with $p=\frac{1}{\min\{\frac{1}{\gamma},1\}}$},\quad\mbox{$\forall x\in\mathfrak{B}(\overline{x},\frac{\eta}{2},\frac{\nu}{N})$.}
\end{equation}
As a consequence, $p=1$ if $\gamma\in (0,1]$ and $p=\gamma$ if $\gamma\in (1,\infty).$
\end{theorem}
\begin{proof}
Since $\overline{\epsilon}<m/L$, $T_{D,\epsilon}(x)\not=\emptyset$. Let $t_p(x)\in Proj_{T_{D,\epsilon}(x)}(x)$.\\
If $F\big{(}t_{p}(x)\big{)}\leq\overline{F}$, then $t_{p}(x)\in[F\leq\overline{F}]$, and we have $$dist(x,[F\leq\overline{F}])\leq\|x-t_{p}(x)\|.$$
The non-trivial case is when
$F\big{(}t_{p}(x)\big{)}>\overline{F}$.
If  $x\in \mathfrak{B}(\overline{x};\frac{\eta}{2},\frac{\nu}{N})\subset\mathfrak{B}(\overline{x};\eta,\nu)$ with
$N$ satisfying the assumptions in
Lemma~3.2, then it is ease to see through the proof of Lemma~3.2 that
 $\|x-t_{p}(x)\|<\frac{\eta}{2}$, and
$t_{p}(x)\in\mathfrak{B}(\overline{x};\eta,\nu)$.
Hence for any $x\in\mathfrak{B}(\overline{x};\frac{\eta}{2},\frac{\nu}{N})$, we have
\begin{eqnarray*}
dist(x,[F\leq\overline{F}])&\leq&\left\{
\begin{array}{ll}
\|x-t_{p}(x)\| &\mbox{if}\quad F\big{(}t_{p}(x)\big{)}\leq\overline{F} \\
\|x-t_{p}(x)\|+dist\big{(}t_{p}(x),[F\leq\overline{F}]\big{)} &\mbox{if}\quad F\big{(}t_{p}(x)\big{)}>\overline{F} \\
\end{array}
\right.\\
&=&\left\{
\begin{array}{ll}
\|x-t_{p}(x)\| &\mbox{if}\quad F\big{(}t_{p}(x)\big{)}\leq\overline{F} \\
\|x-t_{p}(x)\|+c_3^{\frac{1}{\gamma}}dist^{\frac{1}{\gamma}}\big{(}0,\partial_PF\big{(}t_{p}(x)\big{)}\big{)} &\mbox{if}\quad F\big{(}t_{p}(x)\big{)}>\overline{F} \\
\end{array}
\right.\\
&&\mbox{(By the level-set subdifferential error bound condition).}
\end{eqnarray*}
Therefore,  for any $x\in\mathfrak{B}(\overline{x};\frac{\eta}{2},\frac{\nu}{N})$,
by  (ii) of Proposition~\ref{prop:Fk}, we have that
\begin{eqnarray*}
dist(x,[F\leq\overline{F}])&\leq&\|x-t_{p}(x)\|+
c_3^{\frac{1}{\gamma}}(L+\frac{M}{\underline{\epsilon}})^{\frac{1}{\gamma}}\|x-t_{p}(x)\|^{\frac{1}{\gamma}}.\\
&&\qquad\qquad
\end{eqnarray*}
Since $\|x-t_{p}(x)\|<\frac{\eta}{2}$, by the above inequality, we have the following
estimate
\begin{eqnarray*}
    dist(x,[F\leq\overline{F}])&\leq&\left\{
\begin{array}{ll}
\theta_1\|x-t_{p}(x)\| &\mbox{if}\quad 0<\gamma\leq 1, \\
\theta_2\|x-t_{p}(x)\|^{\frac{1}{\gamma}} &\mbox{if}\quad \gamma>1 \\
\end{array}
\right.\\
&=&\theta\|x-t_{p}(x)\|^{\frac{1}{p}}~\quad\\
&=&\theta dist^{\frac{1}{p}}\left(x, T_{D,\epsilon}(x)\right)
\end{eqnarray*}
where $p=\frac{1}{\min\{1,\frac{1}{\gamma}\}}$,
$\theta_1=1+c_3^{\frac{1}{\gamma}}(L+\frac{M}{\underline{\epsilon}})^{\frac{1}{\gamma}}(\frac{\eta}{2})^{\frac{1}{\gamma}-1}$,
$\theta_2=(\frac{\eta}{2})^{1-\frac{1}{\gamma}}+c_3^{\frac{1}{\gamma}}(L+\frac{M}{\underline{\epsilon}})^{\frac{1}{\gamma}}$
 and
$\theta=\max\{\theta_1,\theta_2\}$.
The ``consequence" part follows immediately by the formula $p=\frac{1}{\min\{\frac{1}{\gamma},1\}}$.
\end{proof}

\section{Convergence analysis of VBPG}\label{sec:convergence}
Section 4 and Section 5 discuss convergence behaviors of  sequences  generated by the VBPG method
in Section \ref{sec:Variable Bregman}. For this reason, we will  use $D^k$ {\it explicitly} and will assume that variable Bregman  distances $D^k$ and parameters $\epsilon^k$ satisfy  Assumption~\ref{assump2} uniformly throughout  Sections 4 and 5.
\begin{lemma}\label{lemma1}
Suppose that Assumptions~\ref{assump1} and \ref{assump2} hold and $\epsilon^k\leq\overline{\epsilon}<\frac{m}{L}$ for all $k$. Let $\{x^k\}$ be a sequence generated by the VBPG method. Then the function $F$ satisfies the following decreasing property with positive constant $a$ such that $a=\frac{1}{2}\left(\frac{m}{L}-\overline{\epsilon}\right)$:
\begin{equation}\label{eq:22}
F(x^k)-F(x^{k+1})\geq a\|x^k-x^{k+1}\|^2,\quad\forall k.
\end{equation}
\end{lemma}
\begin{proof}
The claim follows directly from (iv) of Proposition~\ref{prop:Ek} with $a=\frac{1}{2}\left(\frac{m}{L}-\overline{\epsilon}\right)$.
\end{proof}
A number of basic properties of sequences $\{x^k\}$ and $\{F(x^k)\}$ are summarized in the following proposition.
\begin{proposition}\label{proposition4} Suppose that the assumptions of Lemma~\ref{lemma1} hold. Let $\{x^k\}$ be a sequence generated by the VBPG method. Then the following assertions hold:
\begin{itemize}
\item[{\rm(i)}] The sequence $\{F(x^k)\}$ is strictly decreasing (unless $x^k\in\overline{\mathbf{X}}_P$ for some $k$);
\item[{\rm(ii)}] $\sum\limits_{k=0}^{\infty}\|x^k-x^{k+1}\|^2<+\infty$;
\item[{\rm(iii)}] $\{x^k\}$ is bounded, and any cluster point $\overline{x}$ of $\{x^k\}$ is a limiting  critical point of $F$;
that is, $0\in
\partial_L F(\overline{x})$;
\item[{\rm(iv)}] $\lim\limits_{k\rightarrow\infty}dist(x^k,\overline{\mathbf{X}}_L)=0$.
\item[{\rm(v)}]
 If it  is further assumed that $D^k=D$ for all $k$ and $\epsilon^k\rightarrow \hat{\epsilon}\in
(\underline{\epsilon},\frac{m}{2L})$, then any cluster point $\overline{x}$ of $\{x^k\}$ is actually a proximal critical point of $F$: $0\in\partial_P F(\overline{x})$; we also have $\lim\limits_{k\rightarrow\infty}dist\left(x^k,\overline{\mathbf{X}}_P\right)=0$.
\end{itemize}
\end{proposition}
\begin{proof}
{\rm (i):} By Lemma~\ref{lemma1}, we have
\begin{equation}
F(x^{k+1})\leq F(x^k)-a\|x^k-x^{k+1}\|^2.
\end{equation}
If $x^k=x^{k+1}$, then by (iii) of Proposition~\ref{prop:Fk}, we have $x^k\in\overline{\mathbf{X}}_P$. Otherwise $\{F(x^k)\}$ is strictly decreasing and $F(x^k)\rightarrow F_{\zeta}\geq F^*$.\\
{\rm (ii):} By summation for~\eqref{eq:22}, we have
\begin{eqnarray}
a\sum_{k=0}^{N}\|x^k-x^{k+1}\|^2\leq F(x^0)-F(x^{k+1})\leq F(x^0)-F^*.
\end{eqnarray}
Then we obtain
\begin{eqnarray}
\sum_{k=0}^{N}\|x^k-x^{k+1}\|^2\leq\frac{1}{a}[F(x^0)-F^*],
\end{eqnarray}
and it follows that $\sum\limits_{k=0}^{N}\|x^k-x^{k+1}\|^2<+\infty$, $\|x^k-x^{k+1}\|\rightarrow 0$, when $k\rightarrow\infty$.\\
{\rm (iii):} The boundedness of $\{x^k\}$ comes from Assumption~\ref{assump1}, $F=(f+g)$ is level bounded along with the fact that $\{F(x^k)\}$ is strictly decreasing and converges to a finite limit. Since the sequence $\{x^k\}$ is bounded, it has at least one cluster point. Let $\overline{x}$ denote such a point and $x^{k'}\rightarrow\overline{x}$, $k'\rightarrow\infty$.\\
    From statement (ii) of Proposition~\ref{prop:Fk}, we have
$$dist\left(0,\partial_LF(x^{k+1})\right)\leq dist\left(0,\partial_PF(x^{k+1})\right)\leq(L+\frac{M}{\underline{\epsilon}})\|x^k-x^{k+1}\|.$$ Thus from (ii), $\partial_LF\left(x^{k+1}\right)\rightarrow0$ as $k\rightarrow\infty$. Since the graph of $\partial_L F(\cdot)$ is a closed set (Proposition 8.7 of \cite{Rockafellar}) and $x^k\rightarrow \overline{x}$, we have $0\in \partial_L F(\overline{x})$.\\
{\rm (iv):} Suppose this assertion does not hold. Then there exist $\delta>0$, for any $k>0$, we have $k'\geq k$ and $dist(x^{k'},\overline{\mathbf{X}}_L)>\delta$. From the boundedness of $\{x^k\}$, we can assume that $x^{k'}\rightarrow \overline{x}$. Then by (iii) of this proposition, $\overline{x}$ is a limiting critical point. So $dist(\overline{x}, \overline{\mathbf{X}}_L)=0<\delta$, a contradiction. This completes the proof.\\
{\rm (v):} Let $\overline{x}$ be a cluster point of $\{x^k\}$ and $\{x^{k'}\}\subset \{x^k\}$ with $x^{k'}\rightarrow \overline{x}$. Then by (ii), $x^{k'-1}\rightarrow \overline{x}$.
Set $\psi (u,\epsilon,x)=\nabla f(u)^T(x-u)+g(x)+\frac{D(u,x)}{\epsilon}$, and $p(u,\epsilon)=\inf_x \psi (u,\epsilon,x)$. As $D$ and $\epsilon^k$ satisfy Assumption 2, $\mbox{argmin}_x \psi (\overline{x},\hat{\epsilon},x)\not=\emptyset$
by Proposition 1.2 (i).
Let $\bar y\in \mbox{argmin}_x \psi (\overline{x},\hat{\epsilon},x)$. Then
$\psi (u,\epsilon,\bar y)\geq p(u,\epsilon)$ and $\psi (\overline{x},\hat{\epsilon},\bar y)=p(\overline{x},\hat{\epsilon})$.
Since $\psi (\cdot,\cdot,\bar y)$ is continuous at ($\overline{x},\hat{\epsilon}$), it follows that
\begin{eqnarray*}
p(\overline{x},\hat{\epsilon})&=& \psi (\overline{x},\hat{\epsilon},\bar y)=\limsup_{k'\rightarrow +\infty}\psi (x^{k'-1},\epsilon^{k'-1},\bar y)\\
&\geq& \limsup_{k'\rightarrow +\infty} p(x^{k'-1},\epsilon^{k'-1})\\&=&\limsup_{k'\rightarrow +\infty}\psi (x^{k'-1},\epsilon^{k'-1},x^{k'})=\psi (\overline{x},\hat{\epsilon},\overline{x})~~\mbox{(by the continuity of $\psi$ in its domain)}\\
&\geq&  p(\overline{x},\hat{\epsilon}).
\end{eqnarray*}
So $\overline{x}\in  \mbox{argmin}_x \psi (\overline{x},\hat{\epsilon},x)$ and $0\in \partial_P^{x} \psi (\overline{x},\hat{\epsilon},\overline{x})$; that is, $0\in \nabla f(\overline{x})+\partial_P g(\overline{x})=\partial_P F(\overline{x})$. The same argument in the proof of (iv) guarantees that $\lim\limits_{k\rightarrow\infty}dist\left(x^k,\overline{\mathbf{X}}_P\right)=0$.
\end{proof}
\begin{remark}
Note that {\rm{(v)}} holds also if $D^k(x,y)=1/2(y-x)^T Q_k(y-x)$  with $Q_k$ symmetric positive definite matrices and $||Q_k-\hat{Q}||_F\rightarrow 0$ as $k\rightarrow \infty$,where $||\cdot||_F$ is the Frobenius norm. As is evident from the proof,
the key ingredient  of this proof is the continuity assumption on $\psi (\cdot, \cdot,\bar y)$ at the reference point.
\end{remark}
Let $\Omega$ be the set of accumulation points of the sequence $\{x^k\}$ generated by Algorithm. Then $\Omega\neq\emptyset$. We prove in the next proposition that $F$ is actually constant over $\Omega$ if $\underline{\epsilon}\leq\epsilon^k\leq\overline{\epsilon}$. 
\begin{proposition}\label{proposition5}
Suppose the assumptions of Lemma~\ref{lemma1} hold, and $\{x^k\}$ is a sequence generated by the VBPG method.
Let  $\Omega$ be the set of accumulation points of $\{x^k\}$ . Then $F_{\zeta}:=\lim\limits_{k\rightarrow+\infty} F(x^k)$ exists and $F=F_{\zeta}$ on $\Omega$.
\end{proposition}
\begin{proof}
In view of Proposition~\ref{proposition4}, $\{x^{k}\}$ is bounded and $\{F(x^k)\}$ is
a strictly decreasing sequence. So $\lim F(x^k)$ exists and let  $F_{\zeta}$ be the limit.
We now show that $F\equiv F_{\zeta}$ on $\Omega$. Let $\overline{x}\in\Omega$. Then there exists a subsequence
$x^{k'}$ of
$\{x^k\}$ such that  $x^{k'}\rightarrow\overline{x}$. By the continuity of $F$ on \mbox{dom}~$F$ and the convergence  of $\{F (x^k)\}$.
We have
\begin{equation}
F(\overline{x})=\lim_{k'\rightarrow\infty}F(x^{k'})=F_{\zeta}.
\end{equation}
\end{proof}
\section{Convergence rate  analysis of $\{F(x^k)\}$ and $\{x^k\}$  under the level-set subdifferential error bound condition }\label{sec:convergence rate}
In this section we study  the linear rate of  convergence  for the VBPG method under the level-set subdifferential error bound condition  at a point (which depends on $F$ only). Note that the crucial condition really needed is the  level-set proximal Bregman error bound at a point uniformly for all mappings $\{T_{D^k,\epsilon^k}\}$, which depends on
the VBPG method. Thanks to Theorem~\ref{theo:4.2}, the former condition implies the latter condition uniformly for all mappings $\{T_{D^k,\epsilon^k}\}$. We  highlight a fundamental property associated with the function $F$ rather than a property associated with a particular algorithm. In fact, Lemma~\ref{lemma3.1}, Proposition~\ref{prop:3.1} and Theorem~\ref{theo1} still hold under
the level-set Bregman error bound condition uniformly for all mappings  $\{T_{D^{k},\epsilon^k}\}$.


The following lemma provides an upper bound for  the function-value proximity near a critical point under the level-set subdifferential error bound condition.
\begin{lemma}{\bf (Uniform estimate of function-value proximity by Bregman proximal mappings)}\label{lemma3.1}
Suppose that the assumptions of Lemma~\ref{lemma1} hold. Let $\overline{x}\in\overline{\mathbf{X}}_L$.  Suppose that the
level-set subdifferential error bound condition holds at $\overline{x}$
with exponent $\gamma\in(0,1]$, for positive numbers $\eta\in(0,2]$, $\nu$ and $N>\frac{2\overline{\epsilon}\nu}{m-\overline{\epsilon}L}/(\frac{\eta}{2})^2$.
If $x\in\mathfrak{B}(\overline{x};\frac{\eta}{2},\frac{\nu}{N})$, then
 there is a positive number $\kappa'=c_0\theta^2$ such that
\begin{equation}\label{eq:s-3}
F(t_{D^k,\epsilon^k}(x))-\overline{F}\leq\kappa'\|t_{D^k,\epsilon^k}(x)-x\|^2.~~~\mbox{for  all $t_{D^k,\epsilon^k}(x)\in T_{D^k,\epsilon^k}(x)$ and $k=1,2,3,\dots$}
\end{equation}
\end{lemma}
\begin{proof} As $x\in\mathfrak{B}(\overline{x};\frac{\eta}{2},\frac{\nu}{N})$, by Theorem~\ref{theo:4.2}, there is a $\theta$ independent of $k$, such that
$$dist^{2}\left(x,[F\leq\overline{F}]\right)\leq \theta^2 dist^2 (x,T_{D^k,\epsilon^k}(x))\leq \theta^2||t_{D^k,\epsilon^k}(x)-x||^2,$$
for all $t_{D^k,\epsilon^k}(x)\in T_{D^k,\epsilon^k}$ and $k$.
For each $k$, by Proposition~\ref{proposition1.5}, we have
$$F(t_{D^k,\epsilon^k}(x))-\overline{F}\leq c_0dist^{2}\left(x,[F\leq\overline{F}]\right),$$ where
$c_0=\frac{3}{2}L+\frac{M}{2\underline{\epsilon}}$ ($c_0$ is independent of $k$).
Combining the above inequalities yields (\ref{eq:s-3}) with $\kappa'=c_0\theta^2$.
\end{proof}
Under the level-set subdifferential error bound condition, we next show that a sequence generated by
the BVPG method is convergent and has a finite length property.
\begin{proposition}{\bf(Finite length property of sequence $\{x^k\}$)}\label{prop:3.1}
Let the sequence $\{x^k\}$ be generated by the VBPG method and $\overline{x}$ be an accumulation point of $\{x^k\}$, $\overline{F}=F(\overline{x})$. Suppose that the assumptions of Lemma~\ref{lemma1} hold. Assume that the level-set subdifferential error bound holds at the point $\overline{x}$ with exponent $\gamma\in(0,1]$, $\eta\in(0,2]$ and $\nu>0$. Let $a$ and $\kappa'$ be constants given in Lemma~\ref{lemma1} and Lemma~\ref{lemma3.1} respectively.
Let $\sigma\in(0,\frac{\eta}{2})$ and  $\bar{\nu}\in(0,\min\{\frac{\nu}{N},a(\frac{\eta}{2}-\sigma)^2\})$, $N>\frac{2\overline{\epsilon}\nu}{m-\overline{\epsilon}L}/\left(\frac{\eta}{2}\right)^2$. Then the following statements hold.
\begin{itemize}
\item[{\rm(i)}] There is $k_0$ such that $x^k\in\mathfrak{B}(\overline{x};\sigma,\bar{\nu})\subset\mathfrak{B}(\overline{x};\frac{\eta}{2},\frac{\nu}{N})$,~~$\forall k\geq k_0$;
\item[{\rm(ii)}] $\sum\limits_{i=0}^{+\infty}\|x^i-x^{i+1}\|<+\infty$ (finite length property);
\item[{\rm(iii)}] the sequence $\{x^k\}$ actually converges to $\overline{x}$ a limiting critical point of $F$. Moreover $\overline{x}$ is a proximal critical point of $F$ when $D^k=D$, $\epsilon^k\rightarrow\hat{\epsilon}\in\left(\underline{\epsilon},\frac{m}{L}\right)$.
\end{itemize}
\end{proposition}
\begin{proof}
\rm{(i)}: Since $\{F(x^k)\}$ is strictly decreasing, we have $F(x^k)>\overline{F}$, $\forall k$. From assumptions, there is a $k_0$ such that
\begin{eqnarray}
\mbox{{\rm(1)}}&&\overline{F}<F(x^{k_0})<\overline{F}+\overline{\nu};\label{eq:condition1}\\
\mbox{{\rm(2)}}&&\|x^{k_0}-\overline{x}\|+\frac{2(\sqrt{a}+\sqrt{\kappa'})}{a}\sqrt{F(x^{k_0})-\overline{F}}<\sigma.\label{eq:condition2}
\end{eqnarray}
We will use the Principle of Mathematical Introduction to prove that the sequence $\{x^k\}\subset\mathfrak{B}(\overline{x};\sigma,\bar{\nu})$. It is clear that $x^{k_0}\in\mathfrak{B}(\overline{x};\sigma,\bar{\nu})$ by (\ref{eq:condition1}) and (\ref{eq:condition2}). The inequalities $\overline{F}<F(x^{k_0+1})\leq F(x^{k_0})<\overline{F}+\bar{\nu}$ hold trivially. On the other hand, by~\eqref{eq:22}, we have $$\|x^{k_0+1}-x^{k_0}\|\leq\sqrt{\frac{F(x^{k_0})-F(x^{k_0+1})}{a}}\leq\sqrt{\frac{F(x^{k_0})-\overline{F}}{a}}$$ and
$$\|x^{k_0+1}-\overline{x}\|\leq\|x^{k_0}-\overline{x}\|+\|x^{k_0}-x^{k_0+1}\|\leq\|x^{k_0}-\overline{x}\|+\sqrt{\frac{F(x^{k_0})-\overline{F}}{a}}<
\sigma~~(by~\eqref{eq:condition2}).$$
Thus $x^{k_0+1}\in\mathfrak{B}(\overline{x};\sigma,\bar{\nu})$. Now suppose that $x^i\in\mathfrak{B}(\overline{x};\sigma,\bar{\nu})$ for $i=k_0+1,..,k_0+k$ and $x^{k_0+k}\neq x^{k_0+k+1}$. $F(x^{k_0+1})>F(x^{k_0+2})>\cdots>F(x^{k_0+k})>F(x^{k_0+k+1})>\overline{F}$. We need to show that $x^{k_0+k+1}\in\mathfrak{B}(\overline{x};\sigma,\bar{\nu})$. By the concavity of function $h(y)=y^{\frac{1}{2}}$, we have, for $i=k_0+1,k_0+2,\dots,k_0+k$, that
$$
\left(F(x^i)-\overline{F}\right)^{\frac{1}{2}}-\left(F(x^{i+1})-\overline{F}\right)^{\frac{1}{2}}
\geq \frac{1}{2}\frac{[F(x^i)-F(x^{i+1})]}{\left(F(x^i)-\overline{F}\right)^{\frac{1}{2}}}.$$
Recalling that $x^{i+1}\in T_{D^i,\epsilon^i}(x^i)$ and
applying \eqref{eq:22}~ and~\eqref{eq:s-3} to $[F(x^i)-F(x^{i+1})]$ and $(F(x^i)-\overline{F})^{1/2}$ respectively yield
$$\frac{2\sqrt{\kappa'}}{a}||x^i-x^{i-1}||[\left(F(x^i)-\overline{F}\right)^{\frac{1}{2}}-\left(F(x^{i+1})-\overline{F}\right)^{\frac{1}{2}}]\geq ||x^i-x^{i+1}||^2.$$
It follows from $2\sqrt{d_1 d_2}\leq d_1+d_2$ with nonnegative $d_1$ and $d_2$ that
\begin{eqnarray}\label{eq:34}
2\|x^{i+1}-x^i\|\leq\|x^{i}-x^{i-1}\|+\frac{2\sqrt{\kappa'}}{a}\left[\left(F(x^i)-\overline{F}\right)^{\frac{1}{2}}-\left(F(x^{i+1})-\overline{F}\right)^{\frac{1}{2}}\right].
\end{eqnarray}
Summing~\eqref{eq:34} for $i=k_0+1,...,k_0+k$, we obtain
\begin{eqnarray}\label{eq:51}
\sum_{i=k_0+1}^{k_0+k}\|x^{i+1}-x^i\|+\|x^{k_0+k+1}-x^{k_0+k}\|\leq\|x^1-x^0\|+\frac{2\sqrt{\kappa'}}{a}\left[\left(F(x^1)-\overline{F}\right)^{\frac{1}{2}}-\left(F(x^{k_0+k+1})-\overline{F}\right)^{\frac{1}{2}}\right].
\end{eqnarray}
Using  \eqref{eq:51} along with the triangle inequality, we have
\begin{eqnarray*}
\|\overline{x}-x^{k_0+k+1}\|&\leq&\|\overline{x}-x^{k_0}\|+\|x^{k_0}-x^{k_0+1}\|+\sum_{i=k_0+1}^{k_0+k}\|x^{i+1}-x^i\|\\
                   &\leq&\|\overline{x}-x^{k_0}\|+2\|x^{k_0}-x^{k_0+1}\|+\frac{2\sqrt{\kappa'}}{a}\left[\left(F(x^{k_0+1})-\overline{F}\right)^{\frac{1}{2}}\right]\\
                   &\leq&\|\overline{x}-x^{k_0}\|+2\sqrt{\frac{F(x^{k_0})-\overline{F}}{a}}+\frac{2\sqrt{\kappa'}}{a}\left[\left(F(x^{k_0})-\overline{F}\right)^{\frac{1}{2}}\right]\\
&<&\sigma~~~(\mbox{by}~\eqref{eq:condition2}).
\end{eqnarray*}
 This shows that $x^{k_0+k+1}\in\mathfrak{B}(\overline{x};\sigma,\bar{\nu})$, and
(i) is proved by the Principle of Mathematical Induction.\\
\rm{(ii)} and \rm{(iii)}:
A direct consequence of \eqref{eq:51} is, for all $k$,
$$\sum_{i=k_0+1}^{k_0+k}\|x^{i+1}-x^i\|\leq\|x^1-x^0\|+\frac{2\sqrt{\kappa'}}{a}\left[\left(F(x^1)-\overline{F}\right)^{\frac{1}{2}}\right]<+\infty.$$
Therefore
$$\sum_{i=0}^{+\infty}\|x^{i+1}-x^i\|<+\infty.$$
In particular, this implies that the sequence $\{x^k\}$ actually  converges to the point $\overline{x}$.  And $\overline{x}$ is a desired critical point of $F$ by Proposition~\ref{proposition4}.
\end{proof}
\begin{remark} Proposition~\ref{prop:3.1} is still valid under the level-set Bregman error bound condition holding uniformly at point $\overline{x}$ with exponent $p=1$ for all $\{T_{D^k,\epsilon^k}(\cdot)\}$.
\end{remark}
The main result of this section follows.
\begin{theorem}[Sufficient conditions for local linear convergence]\label{theo1} Let a sequence $\{x^k\}$ be generated by the VBPG method, and the sequence $\{x^k\}$ converges to $\overline{x}\in\overline{\mathbf{X}}_L$. Assume that the level-set subdifferential error bound holds at the point $\overline{x}$ with $\gamma\in(0,1]$, $\eta>0$ and $\nu>0$. Suppose that the assumptions of Lemma~\ref{lemma1} hold. Let $\overline{x}\in\overline{\mathbf{X}}_L$ and $\overline{F}=F(\overline{x})$. Let $a$ and $\kappa'$ be constants given in Lemma~\ref{lemma1} and Lemma~\ref{lemma3.1} respectively. Let $\sigma\in(0,\frac{\eta}{2})$ and  $\bar{\nu}\in(0,\min\{\frac{\nu}{N},a(\frac{\eta}{2}-\sigma)^2\})$, $N>\frac{2\overline{\epsilon}\nu}{m-\overline{\epsilon}L}/\left(\frac{\eta}{2}\right)^2$.  Suppose that there is $k_0$ such that $x^{k_0}$ satisfies the conditions~\eqref{eq:condition1}-\eqref{eq:condition2}. Then $\{F(x^k)\}$ converges to value $\overline{F}=F(\overline{x})$ at the $Q$-linear rate of convergence; that is, there are some $\beta\in(0,1)$ and $k_0$ such that
\begin{equation}\label{linearrate}
F(x^{k+1})-\overline{F}\leq\beta(F(x^k)-\overline{F}),\quad\forall k\geq k_0.
\end{equation}
 As a consequence,
\begin{equation}\label{linearrateB}
\sum_{i=1}^{\infty}\left(F(x^{i})-\overline{F}\right)<+\infty.
\end{equation}
Moreover, the sequence $\{x^k\}$  converges at the $R$-linear rate  to a critical point $\hat{x}$; that is, either a limiting critical point or proximal critical point of $F$ (if $D^k=D$, $\epsilon^k\rightarrow\hat{\epsilon}\in\left(\underline{\epsilon},\frac{m}{L}\right)$).
\end{theorem}
\begin{proof}
 For a sequence $\{x^k\}$ generated by the VBPG method, by Proposition~\ref{prop:3.1}, we have that  the sequence $\{F(x^k)\}$ is strictly decreasing, and
converges to $F(\overline{x})=\overline{F}$. Moreover, the sequence $\{x^k\}$ converges to $\overline{x}$, a critical point. In addition, there is $k_0$ such that for $k\geq k_0$ such that $\{x^k\}\subset\mathfrak{B}(\overline{x};\sigma,\bar{\nu})$. For $k\geq k_0$ such that $0<F(x^{k+1})-\overline{F}<1$ and $F(x^k)>\overline{F}$ as $\{F(x^k)\}$ is strictly decreasing, and converges to $F(\overline{x})=\overline{F}$. It follows that
\begin{eqnarray}
F(x^{k+1})-\overline{F}&=&\left(F(x^k)-\overline{F}\right)+\left(F(x^{k+1})-F(x^k)\right)\nonumber\\
                    &\leq&\left(F(x^k)-\overline{F}\right)-a\|x^{k+1}-x^k\|^2\quad\mbox{(by (\ref{eq:22}) in Lemma~\ref{lemma1})}\nonumber\\
                    &\leq&\left(F(x^k)-\overline{F}\right)-a\left(\frac{1}{\kappa'}\right)\left(F(x^{k+1})-\overline{F}\right)\quad\mbox{(by~\eqref{eq:s-3})}
\end{eqnarray}
Therefore
\begin{eqnarray}
F(x^{k+1})-\overline{F}=\frac{1}{1+a\left(\frac{1}{\kappa'}\right)}\left(F(x^k)-\overline{F}\right)\hspace{3mm}\forall k\geq k_0.
\end{eqnarray}
The above estimation shows that $\{F(x^k)\}$ convergences  to $\overline{F}$ at the  Q-linear rate; that is,
\begin{equation}\label{Q-linear}
F(x^{k+1})-\overline{F}\leq \beta\left(F(x^k)-\overline{F}\right)\hspace{3mm}\forall k\geq k_0,
\end{equation}
where $\beta=\frac{1}{1+a\left(\frac{1}{\kappa'}\right)}\in (0,1)$. (\ref{Q-linear}) implies, in particular,
that $ F(x^{k})-\overline{F}\leq\beta^{(k-k_0)}(F(x^{k_0})-\overline{F})$ for all $k\geq k_0$. So
$\sum_{i=k_0}^{\infty}\left(F(x^{k})-\overline{F}\right)<+\infty$ and (\ref{linearrateB}) follows. We now derive  the R-linear rate of convergence of  $\{x^k\}$. By (\ref{eq:22}) in Lemma~\ref{lemma1} again,  we have
\begin{eqnarray}
F(x^k)-F(x^{k+1})\geq a\|x^k-x^{k+1}\|^2.
\end{eqnarray}
Thus
\begin{eqnarray*}
\|x^k-x^{k+1}\|^2&\leq&\frac{1}{a}\bigg{[}\big{(}F(x^k)-\overline{F}\big{)}-\big{(}F(x^{k+1})-\overline{F}\big{)}\bigg{]}\\
                 &\leq&\frac{1}{a}\big{(}F(x^k)-\overline{F}\big{)}\\
                 &\leq&\frac{\beta^{(k-k_0)}}{a}(F(x^{k_0})-\overline{F})~~\mbox{ (by (\ref{Q-linear}))}.
\end{eqnarray*}
From the above inequality, we see that
$$\|x^k-x^{k+1}\|\leq\hat{M}(\sqrt{\beta})^{(k-k_0)}~~~ \forall k> k_0,$$
where $\hat{M}=\sqrt{\frac{F(x^{k_0})-\overline{F}}{a}}$.
By Proposition~\ref{prop:3.1}, we have $\{x^k\}$ converges to desired critical point $\overline{x}$. Hence,
$$\|x^k-\overline{x}\|\leq\sum_{i=k_0}^{\infty}\|x^i-x^{i+1}\|\leq\frac{\hat{M}}{1-\sqrt{\beta}}(\sqrt{\beta})^{(k-k_0)}.$$ This
shows that $\{x^k\}$ converges to desired critical point $\overline{x}$ at the R-linear rate; that is, $$\limsup_{k\rightarrow\infty}\sqrt[(k-k_0)]{\|x^k-\overline{x}\|}=\sqrt{\beta}<1.$$
\end{proof}
\section{Enhanced properties under semi-covexity of $g$ and linear convergence under strong level-set errror bounds}\label{sec:necessary and sufficient conditions}
In this section,  we will again omit $k$ and use the short notation Bregman distance $D$ and positive $\epsilon$ as the conclusions on $D$ hold uniformly regardless of the choice of $k$.
\subsection{The properties of Bregman type mapping and function under semiconvexity of $g$}
\begin{proposition}\label{prop:singlevalue} {\bf (Single-valueness of Bregman proximal mappings)}
Suppose that Assumptions~\ref{assump1} and~\ref{assump2} hold, and that  $g$ is semiconvex on $\RR^n$ with constant $\rho$ and $\overline{\epsilon}<\min\{\frac{m}{L},\frac{m}{\rho}\}$. Then for all $x\in\RR^n$, $T_{D,\epsilon}(x)$ is single-valued.
\end{proposition}
\begin{proof} For each $x$, the nonemptiness of $T_{D,\epsilon}(x)$ follows from (i) of Proposition \ref{prop:Ek}.
Since $g$ is semiconvex, then for $t_{D,\epsilon}^i(x)\in T_{D,\epsilon}(x)$, $i=1,2$, i.e.,
\begin{eqnarray*}
-\left(\nabla f(x)+\frac{1}{\epsilon}\nabla_yD(x,t_{D,\epsilon}^1(x))\right)&=&v^1\in\partial_Pg\left(t_{D,\epsilon}^1(x)\right),\\
-\left(\nabla f(x)+\frac{1}{\epsilon}\nabla_yD(x,t_{D,\epsilon}^2(x))\right)&=&v^2\in\partial_Pg\left(t_{D,\epsilon}^2(x)\right),
\end{eqnarray*}
we have
\begin{eqnarray*}
g(t_{D,\epsilon}^1(x))&\geq&g(t_{D,\epsilon}^2(x))+\langle v^2,t_{D,\epsilon}^1(x)-t_{D,\epsilon}^2(x)\rangle-\frac{\rho}{2}\|t_{D,\epsilon}^1(x)-t_{D,\epsilon}^2(x)\|^2,\\
g(t_{D,\epsilon}^2(x))&\geq&g(t_{D,\epsilon}^1(x))+\langle v^1,t_{D,\epsilon}^2(x)-t_{D,\epsilon}^1(x)\rangle-\frac{\rho}{2}\|t_{D,\epsilon}^1(x)-t_{D,\epsilon}^2(x)\|^2.
\end{eqnarray*}
It follows that
$$\langle v^2-v^1,t_{D,\epsilon}^2(x)-t_{D,\epsilon}^1(x)\rangle\geq-\rho\|t_{D,\epsilon}^1(x)-t_{D,\epsilon}^2(x)\|^2.$$ So
\begin{eqnarray*}
\frac{1}{\epsilon}\langle\nabla_yD(x,t_{D,\epsilon}^1(x))-\nabla_yD(x,t_{D,\epsilon}^2(x)),t_{D,\epsilon}^1(x)-t_{D,\epsilon}^2(x)\rangle-\rho\|t_{D,\epsilon}^1(x)-t_{D,\epsilon}^2(x)\|^2\leq0.
\end{eqnarray*}
By Assumption~\ref{assump2}, $\left(\frac{m}{\overline{\epsilon}}-\rho\right)\|t_{D,\epsilon}^1(x)-t_{D,\epsilon}^2(x)\|^2\leq0$.\\
But $\frac{m}{\overline{\epsilon}}-\rho>0,$ which deduce that $t_{D,\epsilon}^1(x)=t_{D,\epsilon}^2(x)$.
\end{proof}
\begin{proposition}\label{prop:Gk}{\bf (Further properties of Bregman type mappings and functions)}
Suppose that the assumptions of Proposition~\ref{prop:singlevalue} hold. Then for $x\in\RR^n$ and $\overline{\epsilon}<\min\{\frac{m}{L},\frac{m}{\rho}\}$ the following statements hold:
\begin{itemize}
\item[{\rm(i)}] $E_{D,\epsilon}(x)\leq F(x)-\frac{1}{2}\big{(}\frac{m}{\overline{\epsilon}}-\rho\big{)}\|x-T_{D,\epsilon}(x)\|^2$.
\item[{\rm(ii)}] $\frac{1}{2\overline{\epsilon}^2}(m-\overline{\epsilon}\rho)\|x-T_{D,\epsilon}(x)\|^2\leq G_{D,\epsilon}(x)$.
\item[{\rm(iii)}] $G_{D,\epsilon}(x)\leq\frac{1}{2(m-\overline{\epsilon}\rho)}dist^2\left(0,\partial_P F(x)\right)$.
\item[{\rm(iv)}] $\|x-T_{D,\epsilon}(x)\|\leq\left(\frac{\overline{\epsilon}}{m-\overline{\epsilon}\rho}\right)dist\left(0,\partial_PF(x)\right)$.
\item[{\rm(v)}] $G_{D,\epsilon}(x)=0$ if only if $x=T_{D,\epsilon}(x)$ or $0\in\partial_PF(x)$.
\end{itemize}
\end{proposition}
\begin{proof}
From the assumptions and Proposition~\ref{prop:singlevalue}, for $\overline{\epsilon}<\min\{\frac{m}{L},\frac{m}{\rho}\}$, $x\in\RR^n$, we have that $T_{D,\epsilon}(x)$ is single valued.\\
{\rm (i):} From the optimality condition for the minimization problem in~\eqref{APk}, we have
\begin{equation}\label{eq:13}
0\in\nabla f(x)+\partial_L g\big{(}T_{D,\epsilon}(x)\big{)}+\frac{1}{\epsilon}\nabla_{y}D\big{(}x,T_{D,\epsilon}(x)\big{)}
\end{equation}
or
\begin{equation}
-\left(\nabla f(x)+\frac{1}{\epsilon}\nabla_{y}D\big{(}x,T_{D,\epsilon}(x)\big{)}\right)\in\partial_Lg\big{(}T_{D,\epsilon}(x)\big{)},
\end{equation}
Since $g$ is continuous on $\mathbf{dom}g$ and semiconvex with $\rho$,
\begin{eqnarray}\label{eq:10}
g(x)&\geq&g\big{(}T_{D,\epsilon}(x)\big{)}-\frac{\rho}{2}\|x-T_{D,\epsilon}(x)\|^2-\langle\nabla f(x)+\frac{1}{\epsilon}\nabla_{y}D\big{(}x,T_{D,\epsilon}(x)\big{)},x-T_{D,\epsilon}(x)\rangle\;\mbox{(by~\eqref{eq:varphi})}\nonumber\\
&\geq&g\big{(}T_{D,\epsilon}(x)\big{)}-\frac{\rho}{2}\|x-T_{D,\epsilon}(x)\|^2-\langle\nabla f(x),x-T_{D,\epsilon}(x)\rangle\nonumber\\
&&+\frac{1}{\epsilon}D(x,T_{D,\epsilon}(x))-\frac{1}{\epsilon}D(x,x)+\frac{m}{2\overline{\epsilon}}\|x-T_{D,\epsilon}(x)\|^2\nonumber\\
    &\geq&g\big{(}T_{D,\epsilon}(x)\big{)}-\langle\nabla f(x),x-T_{D,\epsilon}(x)\rangle+\frac{1}{\epsilon}D(x,T_{D,\epsilon}(x))+\frac{1}{2}\left(\frac{m}{\overline{\epsilon}}-\rho\right)\|x-T_{D,\epsilon}(x)\|^2.\nonumber\\
    &&\qquad\qquad\qquad\qquad\qquad\qquad\qquad\qquad\qquad\qquad\qquad\qquad\mbox{(by Assumption~\ref{assump2})}
\end{eqnarray}
Adding $f(x)$ to both sides and consider the definition of $E_{D,\epsilon}(x)$ proves the claim.\\
{\rm(ii):} Since $G_{D,\epsilon}(x)=\frac{1}{\epsilon}\big{(}F(x)-E_{D,\epsilon}(x)\big{)}$, from statement (i) of Proposition~\ref{prop:Ek} and (i) of this proposition, we have
\begin{eqnarray}
G_{D,\epsilon}(x)&=&\frac{1}{\epsilon}\big{(}F(x)-E_{D,\epsilon}(x)\big{)}\nonumber\\
      &\geq&\frac{1}{2\overline{\epsilon}^2}\big{(}m-\overline{\epsilon}\rho\big{)}\|x-T_{D,\epsilon}(x)\|^2.
\end{eqnarray}
{\rm(iii):} For $\epsilon<\min\{\frac{m}{L},\frac{m}{\rho}\}$, we have
\begin{eqnarray}
\epsilon G(x)=-\langle\nabla f(x), T_{D,\epsilon}(x)-x\rangle+g\left(T_{D,\epsilon}(x)\right)-g(x)-\frac{1}{\epsilon}D(x,T_{D,\epsilon}(x)).
\end{eqnarray}
Let $\nu\in\partial_Pg(x)$, thanks the semiconvex of $g$, we get
\begin{eqnarray}
\epsilon G(x)&\leq&-\langle\nabla f(x), T_{D,\epsilon}(x)-x\rangle-\langle\nu, T_{D,\epsilon}(x)-x\rangle+\frac{\rho}{2}\|x-T_{D,\epsilon}(x)\|^2-\frac{m}{2\overline{\epsilon}}\|x-T_{D,\epsilon}(x)\|^2\nonumber\\
&=&-\langle\nabla f(x), T_{D,\epsilon}(x)-x\rangle-\frac{1}{2}\left(\frac{m}{\overline{\epsilon}}-\rho\right)\|x-T_{D,\epsilon}(x)\|^2\nonumber\\
&\leq&\|\nabla f(x)+\nu\|\cdot\|x-T_{D,\epsilon}(x)\|-\frac{1}{2}\left(\frac{m}{\overline{\epsilon}}-\rho\right)\|x-T_{D,\epsilon}(x)\|^2\nonumber\\
&\leq&\frac{\overline{\epsilon}}{2(m-\overline{\epsilon}\rho)}\|\nabla f(x)+\nu\|^2.
\end{eqnarray}
Therefore $G_{D,\epsilon}(x)\leq\frac{1}{2(m-\overline{\epsilon}\rho)}\|\nabla f(x)+\nu\|^2$, $\forall\nu\in\partial_Pg(x)$, and the claim is verified.\\
{\rm(iv):} The statement is a simple consequence of (ii) and (iii).
{\rm(v):} The claim follows directly from statements (ii), (iii) and~\eqref{eq:13}.
\end{proof}
\begin{remark}\label{remark2.2}
If $x\in\mathfrak{B}(\overline{x};\frac{\eta}{2},\frac{\nu}{N})$ with $N\geq\frac{2\overline{\epsilon}\nu}{m-\overline{\epsilon}L}/\left(\frac{\eta}{2}\right)^2$ satisfies Property (A), by Lemma ~\ref{lemma:tp}, we have  that $t_{D,\epsilon}(x)\in\mathbb{B}(\overline{x};\eta)$, $\forall t_{D,\epsilon}(x)\in T_{D,\epsilon}(x)$. Furthermore, if $g$ is uniformly proximal regular with $\rho$ and $\eta$ on $\mathbb{B}(\overline{x};\eta)$, then by the definition of uniformly proximal regular function~\eqref{eq:varphi-0}, we conclude that the statements of Proposition~\ref{prop:singlevalue} and~\ref{prop:Gk} still hold on $\mathbb{B}(\overline{x};\eta)$.
\end{remark}
\subsection{The strong level-set error bounds and necessary and sufficient conditions for linear convergence}
Now we introduce the notion of the strong level-set error bounds holding on a  set $[\overline{F}\leq F\leq\overline{F}+\nu]$. This notion along with Proposition~\ref{prop:nec-suf} plays an important role in deriving
a sufficient condition and a necessary condition for linear convergence relative to level sets.
\begin{definition}[Strong Level-set subdifferential error bound] We say that $F$ satisfies the strong level-set subdifferential error bound condition on $[\overline{F}<F<\overline{F}+\nu]$ with the values $\overline{F}$ and $\nu>0$ if there exists $c_3'>0$ such that the following inequality holds:
\begin{equation}\label{LSSEB}
dist(x,[F\leq\overline{F}])\leq c_3'dist\big{(}0,\partial_P F(x)\big{)},~~\forall x\in[\overline{F}<F<\overline{F}+\nu].
\end{equation}
\end{definition}
\begin{definition}[Strong level-set Bregman error bound]
Given a Bregman distance $D$ along with $\epsilon>0$, we say that  $F$ satisfies the strong level-set Bregman proximal (BP) error bound condition on $[\overline{F}<F<\overline{F}+\nu]$ with the values $\overline{F}$ and $\nu>0$ if there exists $\theta'>0$ such that the following inequality holds:
\begin{equation}\label{LSBEB}
dist(x,[F\leq\overline{F}])\leq\theta'dist\left(x,T_{D,\epsilon}(x)\right),~~\forall x\in[\overline{F}<F<\overline{F}+\nu].
\end{equation}
\end{definition}
\begin{corollary}[Strong level-set subdifferential EB $\Rightarrow$ Strong level-set Bregman proximal EB]\label{Corollary:Strong}
Suppose Assumptions~\ref{assump1} and~\ref{assump2} hold with $\overline{\epsilon}<\frac{m}{L}$. Assume the strong level-set subdifferential error bound holds over $[\overline{F}<F<\overline{F}+\nu]$. Then there is $\theta'=1+c_3'(L+\frac{M}{\underline{\epsilon}})>0$ such that
\begin{equation}
dist(x,[F\leq\overline{F}])\leq\theta'dist\left(x,T_{D,\epsilon}(x)\right),~~\forall x\in[\overline{F}<F<\overline{F}+\nu].
\end{equation}
\end{corollary}
\begin{proof}
The claim is proved by the same argument for the proof of Theorem~\ref{theo:4.2} with $\gamma=1$.
\end{proof}

The next proposition will be used to derive a necessary condition and a sufficient condition for linear convergence
with respect to a set.

\begin{proposition}\label{prop:nec-suf}
Suppose Assumptions~\ref{assump1} and~\ref{assump2} hold. Then the following statements hold:
\begin{itemize}
\item[{\rm(i)}] If $F$ satisfies the strong level-set subdifferential error bound condition with $c_3'$ and for given $\overline{F}$, $x\in[\overline{F}<F<\overline{F}+\nu]$ satisfies Property (A), then we have the following  inequality respect to the set $[F\leq\overline{F}]$
    \begin{eqnarray}\label{eq:x-Q}
    dist\left(t_{D,\epsilon}(x),[F\leq\overline{F}]\right)\leq\beta dist\left(x,[F\leq\overline{F}]\right),\quad\forall x\in[\overline{F}<F<\overline{F}+\nu]
    \end{eqnarray}
    with $\beta=\sqrt{\mathfrak{b}-\frac{\mathfrak{c}}{(\theta')^2}}$, $\theta'=1+c_3'(L+\frac{M}{\underline{\epsilon}})$, $\mathfrak{b}$ and $\mathfrak{c}$ are appeared in Lemma~\ref{lemma:1}. Moreover, if $\theta'\in(\sqrt{\frac{c}{b}},\sqrt{\frac{c}{b-1}})$, then $\beta\in(0,1)$.
\item[{\rm(ii)}]  If $g$ is semi-convex on $[\overline{F}<F<\overline{F}+\nu]$ and $\overline{\epsilon}<\min\{\frac{m}{L},\frac{m}{\rho}\}$, then the inequality~\eqref{eq:x-Q} with $\beta\in (0,1)$ implies the strong level-set subdifferential error bound on $[\overline{F}<F<\overline{F}+\nu]$ with $c_3'=\frac{\overline{\epsilon}}{(1-\beta)(m-\overline{\epsilon}\rho)}$.
\end{itemize}
\end{proposition}
\begin{proof}
\begin{itemize}
\item[{\rm(i)}] For $x\in[\overline{F}<F<\overline{F}+\nu]$, let $x_p={Proj}_{[F\leq\overline{F}]}(x)$ in Lemma~\ref{lemma1.5}. Then $F(x_p)=\overline{F}$. By Lemma~\ref{lemma:1} with  $u=x_p$ in~\eqref{eq:descent}, we have $F\left(t_{D,\epsilon}(x)\right)\geq \overline{F}$ and
    \begin{eqnarray}\label{eq:60}
    0\leq\mathfrak{a}[F\left(t_{D,\epsilon}(x)\right)-F(x_p)]\leq\mathfrak{b}\|x_p-x\|^2-\|x_p-t_{D,\epsilon}(x)\|^2-\mathfrak{c}\|x-t_{D,\epsilon}(x)\|^2,
    \end{eqnarray}
    From Corollary~\ref{Corollary:Strong}, the strong level-set subdifferential EB condition implies  the strong level-set Bregman EB condition with $\theta'=1+c_3'(L+\frac{M}{\underline{\epsilon}})$. Thanks to the strong level-set Bregman proximal error bound condition on $[\overline{F}<F<\overline{F}+\nu]$ (see~\eqref{LSBEB}), from~\eqref{eq:60} we have
    \begin{eqnarray}
    \|x_p-t_{D,\epsilon}(x)\|^2&\leq&\mathfrak{b}\|x_p-x\|^2-\mathfrak{c}\|x-t_{D,\epsilon}(x)\|^2\nonumber\\
                               &\leq&\mathfrak{b}\|x_p-x\|^2-\frac{\mathfrak{c}}{(\theta')^2}\|x_p-x\|^2.
    \end{eqnarray}
    and
    \begin{eqnarray}
    dist\left(t_{D,\epsilon}(x),[F\leq\overline{F}]\right)\leq\|x_p-t_{D,\epsilon}(x)\|\leq\left(\mathfrak{b}-\frac{\mathfrak{c}}{(\theta')^2}\right)^{\frac{1}{2}}dist\left(x,[F\leq\overline{F}]\right).
    \end{eqnarray}
\item[{\rm(ii)}]  By semi-convexity of $g$ and $\overline{\epsilon}<\min\{\frac{m}{L},\frac{m}{\rho}\}$, $T_{D,\epsilon}(x)$ is single-valued. Let $T_{D,\epsilon}(x)_p=Proj_{[F\leq\overline{F}]}\left(T_{D,\epsilon}(x)\right)$.  Then we see that
    \begin{eqnarray}\label{TD}
    dist\left(x, [F\leq\overline{F}]\right)&\leq&\|x-T_{D,\epsilon}(x)_p\|\nonumber\\
                                           &\leq&\|T_{D,\epsilon}(x)-T_{D,\epsilon}(x)_p\|+\|x-T_{D,\epsilon}(x)\|\nonumber\\
                                           &=&dist\left(T_{D,\epsilon}(x),[F\leq\overline{F}]\right)+dist\left(x,T_{D,\epsilon}(x)\right)\nonumber\\
                                           &\leq&\beta dist\left(x,[F\leq\overline{F}]\right)+dist\left(x,T_{D,\epsilon}(x)\right).
    \end{eqnarray}
    By the statement (iv) of Proposition~\ref{prop:Gk}, we have
    \begin{eqnarray}
    dist\left(x,[F\leq\overline{F}]\right)&\leq&\frac{1}{(1-\beta)}dist\left(x,T_{D,\epsilon}(x)\right)\nonumber\\
                                          &\leq&\frac{\overline{\epsilon}}{(1-\beta)(m-\overline{\epsilon}\rho)}dist\left(0,\partial_PF(x)\right),
    \end{eqnarray}
    which completes the proof.
\end{itemize}
\end{proof}
The following theorem gives a necessary condition and a sufficient condition for linear convergence relative to a level set.
\begin{theorem}({\bf Necessary and sufficient conditions for linear convergence relative to $[F\leq\overline{F}]$})\label{theo:n-s}
Let a sequence $\{x^k\}$ be generated by the VBPG method, let $\overline{x}$ be an accumulation point of $\{x^k\}$, and let $\nu>0$ be given.
\begin{itemize}
\item[{\rm(i)}] For any  initial point $x^0\in[\overline{F}<F<\overline{F}+\nu]$,  if the strong level-set subdifferential error bound condition holds  on $[\overline{F}<F<\overline{F}+\nu]$ with $c_3'$, then the VBPG method converges linearly respect to level-set $[F\leq\overline{F}]$, i.e.,
\begin{eqnarray}\label{x-Q-linear}
dist\left(x^{k+1},[F\leq\overline{F}]\right)\leq\beta dist\left(x^{k},[F\leq\overline{F}]\right),\quad k\geq0,
\end{eqnarray}
with $\beta=\sqrt{\mathfrak{b}-\frac{\mathfrak{c}}{(\theta')^2}}$ and $\beta\in(0,1)$, $\theta'=1+c_3'(L+\frac{M}{\underline{\epsilon}})\in\left(\sqrt{\frac{\mathfrak{c}}{\mathfrak{b}}},\sqrt{\frac{\mathfrak{c}}{\mathfrak{b}-1}}\right)$, where the values of $\mathfrak{b}$ and $\mathfrak{c}$ are appeared in Lemma~\ref{lemma:1}.
\item[{\rm(ii)}] If $g$ is semi-convex on $\RR^n$, $\overline{\epsilon}<\min\{\frac{m}{L},\frac{m}{\rho}\}$ and the VBPG method converges linearly in the sense of~\eqref{x-Q-linear} with $\beta\in (0,1)$, then $F$ satisfies the strong level-set subdifferential error bound condition~\eqref{LSBEB} on $[\overline{F}<F<\overline{F}+\nu]$ with $c_3'=\frac{\overline{\epsilon}}{(1-\beta)(m-\overline{\epsilon}\rho)}$.
\end{itemize}
\end{theorem}
\begin{proof} Since $\overline{F}=F(\overline{x})$ and  $\{F(x^k)\}$ is strictly decreasing, we have $F(x^k)>\overline{F}$, thus $\forall k,\;x^k$ satisfies Property (A) trivially. The claim follows directly from Proposition~\ref{prop:nec-suf}.
\end{proof}
\begin{remark}
For problem (P), if  $F$ attains the global minimum value $F^*$ at every critical point , then solution set $\mathbf{X}^*=[F\leq F^*]$, $x^k\in[F^*<F<F^*+\nu]$, and  the inequality ~\eqref{x-Q-linear} with respect to $[F\leq F^*]$ becomes
\begin{eqnarray}\label{x-Q-linear-2}
dist\left(x^{k+1},\mathbf{X}^*\right)\leq\beta dist\left(x^k,\mathbf{X}^*\right).
\end{eqnarray}
Observe that a convex or an invex function $F$ satisfies (\ref{x-Q-linear-2}). Furthermore, conditions such as proximal-PL, a global version of KL and proximal EB in~\cite{Schmidt2016} also guarantee (\ref{x-Q-linear-2}).
\end{remark}
\section{Connections with  known error bounds in literature  and applications }\label{sec:level-set type error relationships}
In this section we will examine the relationships of level-set error bounds with existing error bounds. The previous necessary and sufficient condition of linear convergence results of VBPG allow us to exploit the novel convergence results for various existing algorithms. Although  we only study the ``local" version error bounds on $\mathfrak{B}(\overline{x};\eta,\nu)$ in this section, but the same analysis used in this section can be readily extended to  ``global" version error bounds on $[\overline{F}<F<\overline{F}+\nu]$.
\subsection{First type error bounds with target set $\overline{X}_P$ }\label{subsec:error bounds of point}
Let $\overline{x}\in\overline{\mathbf{X}}_P$, we study conditions under which  the distance from any vector $x\in
\mathfrak{B}(\overline{x};\eta,\nu)$ to the set  $\overline{\mathbf{X}}_P$ is bounded by a residual function $R_1(x)$, raised to a certain power, evaluated at $x$.
Specifically, we study the existence of some $\gamma_1$, $\delta_1$, such that
$$dist^{\gamma_1}(x,\overline{\mathbf{X}}_P)\leq\delta_1 R_1(x),\quad\forall x\in\mathfrak{B}(\overline{x};\eta,\nu).$$
An expression of this kind is called a first type error bound with target set $\overline{X}_P$ for (P).

\subsubsection{ Important  examples with target set $\overline{X}_P$}

\begin{definition}[Weak metric-subregularity]  We say that $\partial_P F$ is weakly metrically subregular at $\overline{x} \in \overline{X}_P$ for the zero vector $0$ if there exist $\eta$, $\nu$ and $c_5$ such that
\begin{equation}\label{subregularity}
c_5dist\big{(}x, \overline{\mathbf{X}}_P\big{)}\leq dist\big{(}0,\partial_PF(x)\big{)}, \forall x\in\mathfrak{B}(\overline{x};\eta,\nu).
\end{equation}
\end{definition} A few remarks about (\ref{subregularity}) are in order. Metric subregularity of a set-valued mapping is a well-known notion in
variational analysis. See the monograph \cite{DoR2009} by Dontchev and Rockafellar for motivations, theory, and applications.
In (\ref{subregularity}) if $\mathfrak{B}(\overline{x};\eta,\nu)$ is replaced  by $\mathbb{B}(\overline{x};\eta)$, then
(\ref{subregularity}) is equivalent to
metric subregularity of the set-value mapping $\partial_P F$ at $\bar x$ for the vector $0$ (see Exercise 3H.4 of \cite{DoR2009}) for a proof. Another important notion in variational analysis is calmness of a set-valued mapping. By Theorem 3H.3
of \cite{DoR2009},  metric subregularity of $\partial_P F$ at $\bar x$ for the vector $0$ is equivalent to
the inverse set-valued mapping $(\partial_P F)^{-1}$ is calm at the zero vector $0$ for $\bar x$. In this regard, metric subregularity and calmness can be used to examine properties of a set-valued mapping at a point from two distinct perspectives.  For the set-valued mapping $\partial_P F$, this equivalence can be precisely stated as follows:
\begin{proposition}{\bf (Equivalence of metric subregularity and calmness :
Theorem 3H.3 and Exercise 3H.4 of \cite{DoR2009})}

Let $\partial_P F:\RR^n\rightarrow\RR^n$ be the subdifferential set-valued mapping. Suppose that $0\in \partial_P (\overline{x})$. Then
the following statements are equivalent.

(i) There are $\eta>0$ and $\kappa$ such that
$$(\partial _PF)^{-1} (x^*)\cap \mathfrak{B}(\overline{x};\eta)\subset \overline{X}_P+\kappa\mathbb{B}(0;||x^*||)~\forall x^*\in \RR^n~(\mbox{calmness}).$$

(ii)  There are $\eta>0$ and $\kappa$ such that
 $$dist\big{(}x, \overline{\mathbf{X}}_P\big{)}\leq \kappa dist\big{(}0,\partial_PF(x)\big{)}, \forall x\in\mathbb{B}(\overline{x};\eta)~~(\mbox{metric subregularity}).$$
Furthermore, if (ii) holds, then (\ref{subregularity}) holds; that is, $\partial_P F$ is weakly mertric-subregular at
$\overline{x}$ for the zero vector $0$.
\end{proposition}

\begin{definition}[Bregman proximal error bound]
Given a Bregman function $D$ along with $\epsilon>0$, we say that the Bregman proximal error bound holds
at $\overline{x} \in \overline{X}_P$  if there exist $\eta$, $\nu$ and $c_4$ such that
\begin{equation}\label{Bregman proximal error bound}
dist(x,\overline{\mathbf{X}}_P)\leq c_4dist\big{(}x, T_{D,\epsilon}(x)\big{)}, \forall x\in\mathfrak{B}(\overline{x};\eta,\nu).
\end{equation}
\end{definition}
\begin{assumption} [H$_3$]\label{assump4}
There is a $\delta>0$ such that $F(y)\leq F(\overline{x})$ whenever $y\in\overline{\mathbf{X}}_P$ and $\|y-\overline{x}\|\leq\delta$.
\end{assumption}
Next theorem states that, for $F$, if the Bregman proximal error bound or  $\partial_P F$ is weakly metric-subregular at $\overline{x}\in\overline{\mathbf{X}}_P$ for the zero vector $0$, and {\bf(H$_3$)} holds at $\overline{x}$, then the level-set subdifferential error bound  holds at $\overline{x}$.
\begin{theorem}\label{theorem:subregularity}{\bf (Level-set subdifferential error bound under the Bregman proximal error bound or weak metric-subregularity )} Suppose  Assumption~\ref{assump1} holds, and Assumption {\bf(H$_3$)} holds
at $\overline{x}\in\overline{\mathbf{X}}_P$. If one of the following condition holds
\begin{itemize}
\item[{\rm(i)}] the Bregman proximal error bound  holds uniformly at $\overline{x}$, and $g$ is semiconvex or $g$ is uniformly prox-regular around $\overline{x}$ with $\rho$, $\eta$, $\overline{\epsilon}<\min\{\frac{m}{L},\frac{m}{\rho}\}$, and $x$  satisfies Property {\rm(A)}
for  $x\in\mathfrak{B}(\overline{x};\frac{\eta}{2},\frac{\nu}{N})$ with $N\geq\frac{2\overline{\epsilon}\nu}{m-\overline{\epsilon}L}/\left(\frac{\eta}{2}\right)^2$.
\item[{\rm(ii)}] $\partial_P F$ is weakly metric-subregular at $\overline{x}$ for the zero vector $0$;
\end{itemize}
then  the level-set subdifferential error bound holds at $\overline{x}$ with
$\gamma=1$.
\end{theorem}
\begin{proof}
First noted that $\overline{\mathbf{X}}_P=\{\overline{\mathbf{X}}_P\cap\mathfrak{B}(\overline{x};\eta,\nu)\}\cup\{\overline{\mathbf{X}}_P\setminus \mathfrak{B}(\overline{x};\eta,\nu)\}$. For given $x\in\mathfrak{B}(\overline{x};\frac{\eta}{2},\nu)$, we have $dist\left(x,\overline{\mathbf{X}}_P\cap\mathfrak{B}(\overline{x};\eta,\nu)\right)\leq\frac{\eta}{2}$ and $dist\left(x,\overline{\mathbf{X}}_P\setminus\mathfrak{B}(\overline{x};\eta,\nu)\right)>\frac{\eta}{2}$. Let $x_p={Proj}_{\overline{X}_P}(x)$, then we must have $x_p\in\overline{\mathbf{X}}_P\cap\mathfrak{B}(\overline{x};\eta,\nu)$ and $dist\left(x,\overline{\mathbf{X}}_P\right)=dist\left(x,\overline{\mathbf{X}}_P\cap\mathfrak{B}(\overline{x};\eta,\nu)\right)$.\\
{\rm(i):} By Assumption {\bf(H$_3$)}, for $\eta\leq\delta$, we have that $\overline{\mathbf{X}}_P\cap\mathfrak{B}(\overline{x};\eta,\nu)\subset[F\leq\overline{F}]$. Therefore, for $x\in\mathfrak{B}(\overline{x};\frac{\eta}{2},\nu)\subset\mathfrak{B}(\overline{x};\eta,\nu)$, we conclude
\begin{eqnarray}\label{eq:71}
dist\left(x,[F\leq\overline{F}]\right)&\leq& dist\left(x,\overline{\mathbf{X}}_P\cap\mathfrak{B}(\overline{x};\eta,\nu)\right)\nonumber\\
&=&dist(x,\overline{\mathbf{X}}_P)\nonumber\\
&\leq&c_4 dist\big{(}x,T_{D,\epsilon}(x)\big{)}\\
&&\qquad\mbox{(by the definition of Bregman proximal error bound~\eqref{Bregman proximal error bound})}\nonumber.
\end{eqnarray}
If $g$ is semi-convex, Proposition~\ref{prop:Gk} is applicable.  For the case where $g$ is uniformly prox-regular around $\overline{x}$ with $\rho$, $\eta$, $\overline{\epsilon}<\min\{\frac{m}{L},\frac{m}{\rho}\}$ and $x\in\mathfrak{B}(\overline{x};\frac{\eta}{2},\frac{\nu}{N})$ with $N\geq\frac{2\overline{\epsilon}\nu}{m-\overline{\epsilon}L}/\left(\frac{\eta}{2}\right)^2$ satisfying Property (A), $T_{D,\epsilon}(x)$ is single valued and the statement (iv) of Proposition~\ref{prop:Gk} is still valid on $\mathbb{B}(\overline{x};\eta)$. Moreover,  we  have $dist\big{(}x,T_{D,\epsilon}(x)\big{)}\leq\frac{\overline{\epsilon}}{(m-\overline{\epsilon}\rho)}dist \left(0,\partial_P F(x)\right)$. By~\eqref{eq:71}, the level-set subdifferential error bound holds at $\overline{x}\in\overline{\mathbf{X}}_P$.\\
{\rm(ii):} By Assumption {\bf(H$_3$)}, for $\eta\leq\delta$, we have that $\overline{\mathbf{X}}_P\cap\mathbb{B}(\overline{x};\eta)\subset[F\leq\overline{F}]$. For $x\in\mathfrak{B}(\overline{x};\frac{\eta}{2},\nu)$, since $\partial_P F$ satisfies weak metric subregularity, we have
\begin{eqnarray}
dist\big{(}x,\partial_PF(x)\big{)}&\geq& c_5dist\big{(}x, \overline{\mathbf{X}}_P\big{)}\nonumber\\
&=&c_5dist\left(x,\overline{\mathbf{X}}_P\cap\mathfrak{B}(\overline{x};\eta,\nu)\right)\nonumber\\
&\geq& c_5dist\left(x,[F\leq\overline{F}]\right),
\end{eqnarray}
which yields the desired result.
\end{proof}
\subsubsection{Convergence analysis of various algorithms under first type error bounds for  linear convergence}
{\bf Application 7.1: Linear convergence of PG method for fully nonconvex problem (P)}\\
Very recently, Wang et. al.~\cite{Ye18} develop the perturbation technique to conducting the linear convergence of the PG method under the calmness condition along with the proper separation of stationary value condition {\bf(H$_3$)} for fully nonconvex problem (P). From Theorem~\ref{theo1}, we see that linear convergence of PG is in fact guaranteed by the level-set subdifferential error bound condition which is weaker than the calmness condition don't required Assumption~{\bf(H$_3$)}.
Moreover, if $g$ is semi-convex, the strong level-set subdifferential error bound on $[\overline{F}<F<\overline{F}+\nu]$ is necessary and sufficient for linear convergence in sense~\eqref{x-Q-linear}.\\

{\bf Application 7.2: Linear convergence of regularized Jaccobi method}\\
In many big data applications, the regularizer $g$ in problem (P) may have block separable structures, i.e., $g(x)=\sum\limits_{i=1}^{N}g_i(x_i)$, $x_i\in\RR_n^i$. In this setting, (P) can be specified as
\begin{equation}\label{separable}
\min_{x\in\RR^n} f(x_1,...,x_n)+\sum_{i=1}^{N}g_i(x_i)
\end{equation}
If we take $K^k(x)=\sum\limits_{i=1}^{N}f\left(R_i^k(x)\right)+\frac{c_i}{2}\|x_i-x_i^k\|^2$ and $D^k(x,y)=K^k(y)-\left[K^k(x)+\langle\nabla K^k(x),y-x\rangle\right]$, where $R_i^k\triangleq(x_1^k,...,x_{i-1}^k,x_i,x_{i+1}^k,...,x_n^k)$. Thus VBPG become a regularized Jaccobi algorithm. Recently, G. Bajac \cite{Banjac18} provided the linear convergence of regularized Jaccobi algorithm under quadratic growth condition for full convex problem~\eqref{separable}. From the results of in the convex setting quadratic growth condition is equivalent to Bregman proximal error bound, metric subregularity and KL property with exponent $\frac{1}{2}$, see \cite{Lewis2018} and \cite{YeJ18} for more details.\\
By Theorem~\ref{theo1},  for full nonconvex problem~\eqref{separable}, the VBPG method provides the linear convergence under the  level-set subdifferential error bound condition at the point $\overline{x}\in\overline{\mathbf{X}}_L$. For the convex problem~\eqref{separable}, together with Theorem~\ref{theo:n-s}, we can show that the quadratic growth condition is also necessary for linear convergence in the sense of~\eqref{x-Q-linear-2}.
\subsection{Second type error bounds with target value $F(\overline{x})$ }\label{subsec:error bounds of function}
 Second type error bounds are used to bound the absolute difference  of  any function value $F$  at $\overline{x}\in\overline{\mathbf{X}}_P$ from a test set to  the value $\overline{F}=F(\overline{x})$ by a residual function $R_2$. Specifically we study if there exist some $\gamma_2$, $\delta_2$ such that
$$R_2(x)\geq\delta_2\left(F(x)-\overline{F}\right)^{\gamma_2},\quad\forall x\in\mathfrak{B}(\overline{x};\eta,\nu).$$
An expression of this kind is called a second type error bound of $F$ for problem (P).
\subsubsection{ Connections of important second type error bounds and  level-set based error bounds}
\begin{definition}[Kurdyka-{\L}ojasiewicz property] The proper lower semicontinuous function $F$ is said to satisfy the Kurdyka-{\L}ojasiewicz  (KL) property at $\overline{x}$ with exponent $\alpha\in(0,1)$, if there exist $\nu>0$, $\eta>0$, and $c_1>0$ such that the following inequality holds:
\begin{equation}\label{eq:KL_1}
dist(0,\partial_LF(x))\geq c_1[F(x)-\overline{F}]^{\alpha}~~~\forall x\in\mathfrak{B}(\overline{x};\eta,\nu).
\end{equation}
\end{definition}
\begin{definition}[Level-set sharpness]  A proper lower semicontinuous function $\psi$ is said to satisfy the level-set sharpness at $\overline{x}$ with exponent $\beta>0$, if there exist $\nu>0$,$\eta>0$, and $c_2>0$ such that the following inequality holds:
$$dist(x,[F \leq \overline{F}])\leq c_2\big{(}F(x)-\overline{F}\big{)}^{\beta}~~~\forall x\in\mathfrak{B}
(\overline{x};\eta,\nu).$$
\end{definition}

Generally speaking, the KL property is the strongest property that implies all others.
\begin{theorem}[ Level-set subdifferential EB under KL property]\label{kl}
 Let $F$ be a proper lower semicontinuous function on $R^n$. Suppose that
$F$ satisfies the KL property  at $\bar x$ with  exponent  $\alpha\in(0,1)$ over $\mathfrak{B}(\overline{x};\eta,\nu)$. Then
\begin{itemize}
\item[\rm{(a)}] The function $F$ is level-set sharpness at $\overline{x}$ with $\beta=1-\alpha$  over $x\in\mathfrak{B}(\overline{x};\frac{\eta}{2},\nu)$.
\item[\rm{(b)}] Moreover, $F$ also satisfies the level-set subdifferential error bound at $\overline{x}$ with $\gamma=\frac{\alpha}{1-\alpha}$ over $x\in\mathfrak{B}(\overline{x};\frac{\eta}{2},\nu)$.
So $\alpha=\frac{\gamma}{1+\gamma}$. As a consequence,
 $\alpha\in(0,1/2]$ if and only if $\gamma\in(0,1]$, and $\alpha\in (\frac{1}{2},1)$ if and only if $\gamma>1$.
\end{itemize}
\end{theorem}
\begin{proof}
\begin{itemize}
\item[\rm{(a)}] See Theorem 4.1 of Az\'e and Corvellec with the name nonlinear local error bound for level-set sharpness condition~\cite{Aze2017}.
\item[\rm{(b)}] From (a), there is some $c_2>0$ such that for $x\in\mathfrak{B}(\overline{x};\frac{\eta}{2},\nu)$ we have
$$dist(x,[F\leq \overline{F}])\leq c_2\big{(}F(x)-\overline{F}\big{)}^{1-\alpha}.$$
Then as the KL  property holds at $\bar x$ with exponent $\alpha\in (0,1)$, there is some $c_1>0$ such that
\begin{eqnarray*}
dist(x,[F\leq \overline{F}])&\leq&c_2\frac{1}{c_{1}^{\frac{1-\alpha}{\alpha}}}dist^{\left(\frac{1-\alpha}{\alpha}\right)}\big{(}0,\partial_L F(x)\big{)}\\
&=&\frac{c_1c_2}{{c_1}^{\frac{1}{\alpha}}}dist^{\left(\frac{1-\alpha}{\alpha}\right)}\big{(}0,\partial_L F(x)\big{)}.
\end{eqnarray*}
This yields $\gamma=\frac{\alpha}{1-\alpha}$.
Since $\partial_PF(x)\subseteq\partial_LF(x)$, the claim is proved. The last part follows easily with some simple computations.
\end{itemize}
\end{proof}

The following proposition reveals that the notion of level-set subdifferential EB condition is actually weaker than that of KL property.
\begin{proposition}
For any real number $\overline{x}$,
there is a lower-semi continuous function $\psi:R\rightarrow R$ such that if $\mathfrak{B}(\overline{x};\eta,\nu)\not=\emptyset$ with $\eta>0$ and $\nu>0$, then for $\overline{\psi}=\psi(\overline{x})$,
the level-set subdifferential EB condition holds at $\overline{x}$, but KL property fails at $\overline{x}$.
\end{proposition}
\begin{proof} For a given real number $\overline{x}$, let $\psi:R\rightarrow R$ be given by $\psi (x)=1/2(x-\overline{x})^2$ if $x\not=\overline{x}$ and $\psi(x)=-1$ if $x=\overline{x}$.
Then $\partial_P \psi (x)=\{(x-\overline{x})\}$ if $x\not=\overline{x}$. Suppose
that $\mathfrak{B}(\overline{x};\eta,\nu)\not=\emptyset$ with $\eta>0$ and $\nu>0$. It follows that $dist(x,[\psi\leq \overline{\psi}])=|x-\overline{x}|\leq
dist(0,\partial_P(x))$ for all $x\in \mathfrak{B}(\overline{x};\eta,\nu)$. So the level-set subdifferential EB condition holds at $\overline{x}$ with
$\gamma=1$ and $c_3=1$ by Definition 3.1. However $\psi (x)-\psi (\overline{x})\geq 1$ for any $x\in \mathfrak{B}(\overline{x};\eta,\nu)$, and $dist (0,\partial_L\psi(x))=|x-\overline{x}|\rightarrow 0$ as $|x-\overline{x}|\rightarrow 0$.
Hence KL property fails at $\overline{x}$ for any $\alpha\in (0,1)$ (see Definition 7.3).
\end{proof}
\begin{definition}[Bregman proximal gap condition]
Given  a Bregman function  $D$ along with  $\epsilon>0$, we say that  the function $F$  satisfies the
Bregman proximal  (BP) gap condition relative to $D$ and $\epsilon$  at $\overline{x}\in{\rm dom}F$ with exponent $q\in[0,2)$ if there exist $\nu>0$, $\eta>0$, and $\mu>0$ such that the following inequality holds:
$$G_{D,\epsilon}(x)\geq\mu\big{(}F(x)-\overline{F}\big{)}^q~~~\quad\forall x\in\mathfrak{B}(\overline{x};\eta,\nu),$$
where $G_{D,\epsilon}(x)=-\frac{1}{\epsilon}\min\limits_{y\in\RR^n}\big{\{}\langle\nabla f(x),y-x\rangle+g(y)-g(x)+\frac{1}{\epsilon}D(x,y)\big{\}}$.
\end{definition}
Under the assumption of uniform prox-regularity of $g$ at $\overline{x}$, we have the following theorem, which gives
an answer to  the converse of Theorem~\ref{kl}.
\begin{theorem} \label{theo:4.3}{\bf(BP gap condition, level-set Bregman EB and KL property)}
Suppose that Assumption 1 holds.
For a given  Bregman function $D$ along with $\epsilon>0$ satisfying Assumption~\ref{assump2}, $g$ is semi-convex or $g$ is uniformly prox-regular around $\overline{x}$ with $\rho$, $\eta$, $\overline{\epsilon}<\min\{\frac{m}{L},\frac{m}{\rho}\}$ and  $x\in\mathfrak{B}(\overline{x};\frac{\eta}{2},\frac{\nu}{N})$ with $N\geq\frac{2\overline{\epsilon}\nu}{m-\overline{\epsilon}L}/\left(\frac{\eta}{2}\right)^2$ satisfies Property (A):
\begin{itemize}
\item[\rm{(i)}] If $F$ satisfies level-set Bregman error bound holds at $\overline{x}$ with exponent $p$ over $\mathfrak{B}(\overline{x};\eta,\nu)$, then BP gap condition holds at $\overline{x}$  with exponent $q=\frac{1}{\min\{\frac{1}{p},1\}}$
over $\mathfrak{B}(\overline{x};\frac{\eta}{2},\frac{\nu}{N})$ with $N\geq\frac{2\overline{\epsilon}\nu}{m-\overline{\epsilon}L}/\left(\frac{\eta}{2}\right)^2$.
\item[\rm{(ii)}] If $F$ satisfies BP gap condition at $\overline{x}$ with exponent $q$ over $\mathfrak{B}(\overline{x};\eta,\nu)$, then function $F$ has the KL property at $\overline{x}$ with exponent of $\frac{q}{2}$ over $\mathfrak{B}(\overline{x};\eta,\nu)$ .
\end{itemize}
\end{theorem}
\begin{proof}
\rm{(i):} For $x\in\mathfrak{B}(\overline{x};\frac{\eta}{2},\frac{\nu}{N})$, let $x_p\in[F\leq\overline{F}]$ s.t. $\|x-x_p\|=dist(x,[F\leq\overline{F}])$. By Lemma~\ref{lemma1.5}, we have $F(x_p)=F(\overline{x})=\overline{F}$ and the estimate for term $E_{D,\epsilon}(x)-\overline{F}$ can obtained by Proposition~\ref{proposition1.5} as following
    $$E_{D,\epsilon}(x)-\overline{F}\leq c_0{dist}^2(x,[F\leq\overline{F}]),\quad\mbox{with}\quad c_0=\frac{3}{2}L+\frac{M}{2\underline{\epsilon}}.$$
Furthermore, we obtain
\begin{eqnarray*}
F(x)-\overline{F}&=&F(x)-E_{D,\epsilon}(x)+E_{D,\epsilon}(x)-\overline{F}\\
                 &\leq&F(x)-E_{D,\epsilon}(x)+c_0dist^2(x,[F\leq \overline{F}])\qquad\mbox{(by~\eqref{eq:4.1})}\\
                 &\leq&\epsilon G_{D,\epsilon}(x)+c_0\theta^2dist^{\frac{2}{p}}\left(x,T_{D,\epsilon}(x)\right)\;\mbox{(by level-set Bregman EB condition)}\\
                 &\leq&\epsilon G_{D,\epsilon}(x)+c_0\theta^2\left(\frac{2\overline{\epsilon}^2}{m-\overline{\epsilon}\rho}\right)^{\frac{1}{p}}\left(G_{D,\epsilon}(x)\right)^{\frac{1}{p}}\\
                 &&\quad\mbox{(by (ii) of Proposition~\ref{prop:Gk} is valid)}
\end{eqnarray*}
So, there is some $\mu>0$ such that
$$G_{D,\epsilon}(x)\geq\mu\left(F(x)-\overline{F}\right)^q,\quad\forall x\in\mathfrak{B}(\overline{x};\frac{\eta}{2},\frac{\nu}{N}),\quad q=\frac{1}{\min\{\frac{1}{p},1\}}.$$
The proof is completed.\\
\rm{(ii):} By the hypothesis, the BP gap condition  holds at $\overline{x}$ with exponent $q$
over $\mathfrak{B}(\overline{x};\eta,\nu)$ , i.e.,
$$G_{D,\epsilon}(x)\geq\mu\left(F(x)-\overline{F}\right)^q~~\forall x\in \mathfrak{B}(\overline{x};\eta,\nu). $$
By the assumptions for $g$, then $\partial_Pg(x)=\partial_Lg(x)$ and $\partial_P F(x)=\partial_L F(x)$, $\forall
x\in \mathfrak{B}(\overline{x};\frac{\eta}{2},\frac{\nu}{N})$
the statement (iv) of Proposition~\ref{prop:Gk} is valid, we have
\begin{equation}
G_{D,\epsilon}(x)\leq\frac{dist^2(0,\partial_LF(x))}{2(m-\overline{\epsilon}\rho)}.
\end{equation}
It follows that
\begin{equation}
2(m-\overline{\epsilon}\rho)\mu\left(F(x)-\overline{F}\right)^q\leq\left[dist\big{(}0,\partial_L F(x)\big{)}\right]^2.
\end{equation}
Thus
\begin{equation}
dist\big{(}0,\partial_L F(x)\big{)}\geq\sqrt{2(m-\overline{\epsilon}\rho)\mu}\left(F(x)-\overline{F}\right)^{\frac{q}{2}},
\end{equation} and the assertion is justified.
\end{proof}

Figure 1 summarizes the main results of this section.
\begin{figure}
\centering
\begin{center}
\scriptsize
		\tikzstyle{format}=[rectangle,draw,thin,fill=white]
		\tikzstyle{test}=[diamond,aspect=2,draw,thin]
		\tikzstyle{point}=[coordinate,on grid,]
\begin{tikzpicture}
[
>=latex,
node distance=5mm,
 ract/.style={draw=blue!50, fill=blue!5,rectangle,minimum size=6mm, very thick, font=\itshape, align=center},
 racc/.style={rectangle, align=center},
 ractm/.style={draw=red!100, fill=red!5,rectangle,minimum size=6mm, very thick, font=\itshape, align=center},
 cirl/.style={draw, fill=yellow!20,circle,   minimum size=6mm, very thick, font=\itshape, align=center},
 raco/.style={draw=green!500, fill=green!5,rectangle,rounded corners=2mm,  minimum size=6mm, very thick, font=\itshape, align=center},
 hv path/.style={to path={-| (\tikztotarget)}},
 vh path/.style={to path={|- (\tikztotarget)}},
 skip loop/.style={to path={-- ++(0,#1) -| (\tikztotarget)}},
 vskip loop/.style={to path={-- ++(#1,0) |- (\tikztotarget)}}]

        \node (a) [ractm, xshift=-20]{\baselineskip=3pt\small {\bf level-set subdifferential}\\
                          \baselineskip=3pt\small {\bf error bound}\\ \baselineskip=3pt\footnotesize$dist^{\gamma}\left(x,[F\leq\overline{F}]\right)\leq c_3{dist}\left(0,\partial_PF(x)\right)$};
        \node (b) [ract, below = of a, xshift=20, yshift=-20]{\baselineskip=3pt\small {\bf level-set Bregman error bound}\\ \baselineskip=3pt\footnotesize$dist^p\left(x,[F\leq\overline{F}]\right)\leq\theta dist\left(x,T_{D,\epsilon}(x)\right)$\\
        \baselineskip=3pt\footnotesize$\forall x\in\mathfrak{B}(\overline{x};\eta,\nu)$};
        \node (b1) [ract, right = of b, xshift=20]{\baselineskip=3pt\small {\bf BP gap condition}\\
                          \baselineskip=3pt\footnotesize$G_{D,\epsilon}\left(x\right)\geq\mu\left(F(x)-\overline{F}\right)^{q}$};
        \node (bbb) [racc, below= of b, xshift=100, yshift=32]{\baselineskip=3pt\footnotesize$g$ is\\
                                                              \baselineskip=3pt\footnotesize semi-convex};
        \node (aa1) [cirl, above = of a]{$+$};
        \node (aa2) [cirl, left = of a]{$+$};
        \node (aa2a) [racc, below= of aa2, yshift=12]{\baselineskip=3pt\footnotesize $g$ is\\
                                               \baselineskip=3pt\footnotesize semi-convex};
        \node (aba) [above = of aa2, yshift=12] {\baselineskip=3pt\footnotesize {\bf(H$_3$)}};
        \node (aba1) [racc, above= of aba, xshift=-15, yshift=-8]{\baselineskip=3pt\footnotesize $g$ is semi-convex};
        \node (aa3) [ract, left = of aa2]{\baselineskip=3pt\small {\bf Bregman proximal}\\
        \baselineskip=3pt\small {\bf error bound}\\
        \baselineskip=3pt\footnotesize$dist\left(x,\overline{\mathbf{X}}_P\right)\leq c_4dist\left(x,T_{D,\epsilon}(x)\right)$};
        \node (aa3-1) [ract, above = of aa3, xshift=-45]{\baselineskip=3pt\footnotesize{\bf Calmness}\\
                                                         \baselineskip=3pt\footnotesize{\bf at} $(0,\overline{x})$.};
        \node (aa4) [ract, above = of aa1]{\baselineskip=3pt\small {\bf weak metric subregularity}\\ \baselineskip=3pt\footnotesize$dist\left(0,\partial_PF(x)\right)\geq c_5dist\left(x,\overline{\mathbf{X}}_P\right)$\\
        \baselineskip=3pt\footnotesize\qquad\qquad\qquad\qquad$\forall x\in\mathfrak{B}(\overline{x};\eta,\nu)$};
        \node (aa5) [ract, left = of aa4, xshift=-40, yshift=30]{\baselineskip=3pt\small {\bf metric subregularity}\\ \baselineskip=3pt\footnotesize$dist\left(0,\partial_PF(x)\right)\geq c_5dist\left(x,\overline{\mathbf{X}}_P\right)$\\
        \baselineskip=3pt\footnotesize\qquad\qquad\qquad\qquad$\forall x\in\mathbb{B}(\overline{x},\eta)$};
        \node (dd)[ract, below = of b,xshift=-112, yshift=-15]{\baselineskip=3pt\small $\{F(x^k)\}$ Q-linear\\
                                         \baselineskip=3pt\small $\{x^k\}$ R-linear};
        \node (ddd)[ract, right = of dd, xshift=-16]{\baselineskip=3pt\small $\{dist(x^k,[F\leq\hat{F}])\}$\\
                                       \baselineskip=3pt\small Q-linear};
        \node (dddd)[racc, left = of dd, xshift=10]{\baselineskip=3pt\small Linear convergence for VBPG};
        \node (ddddd)[ract, right = of ddd, xshift=10]{\baselineskip=3pt\small {\bf strong level-set}\\
                                                        \baselineskip=3pt\small {\bf subdifferential error bound}\\ \baselineskip=3pt\footnotesize$dist\left(x,[F\leq\overline{F}]\right)\leq c_3'dist\left(x,\partial_PF(x)\right)$\\
        \baselineskip=3pt\footnotesize$\forall x\in[\overline{F}< F<\overline{F}+\nu]$};
        \node (g) [ract, right = of a, xshift=3]{\baselineskip=3pt\small {\bf KL property}\\ \baselineskip=3pt\footnotesize $dist\left(0,\partial_LF(x)\right)\geq c_1\left(F(x)-\overline{F}\right)^{\alpha}$};
        \node (hh) [racc, below = of b, xshift=205, yshift=90]{\baselineskip=3pt\footnotesize $g$ is\\
                                                               \baselineskip=3pt\footnotesize semi-convex};
        \path 
              (b) edge[->] (b1)
              (g) edge[->] (a)
            (aa3) edge[->] (aa2)
            (aa2) edge[->] (a)
            (aa4) edge[->] (aa1)
            (aa1) edge[->] (a)
            (aba) edge[->] (aa2)
            (aba) edge[->] (aa1)
             (b1) edge[->] (5.95,-0.5)
          (aa3-1) edge[->] (-9.25,0.7)
          (aa3-1) edge[->] (-9.25,3.4)
          (-9.25,3.4) edge[->] (aa3-1);
        \path ( 1.5, -0.7) edge[->] ( 1.5, -1.9);
        \path (-3.1, -0.7) edge[->] (-3.1, -4.3);
        \path ( 0.65, -4.7) edge[->] (1.5, -4.7);
        \path ( 1.5, -5.1) edge[->] (0.65, -5.1);
        \path (aa5) edge[->,hv path] (aa4);
        \path (-6.65,0.7) edge[->, vh path] (-3.15,2.55);
\end{tikzpicture}
\caption{The relationship among the notions of the level-set subdifferential EB, subregularity of subdifferential, Bregman proximal EB, KL property, level-set Bregman EB and Bregman gap condition \big{(}In this Figure, conditions of semiconvexity of $g$ can be replaced by uniformly proximal regular with $\rho$ and $\eta$ and $x\in\mathfrak{B}(\overline{x};\frac{\eta}{2},\frac{\nu}{N})$ with $N\geq\frac{2\overline{\epsilon}\nu}{m-\overline{\epsilon}L}/\left(\frac{\eta}{2}\right)^2$ satisfies Property (A)\big{)}}\label{fig:1}
\end{center}
\end{figure}
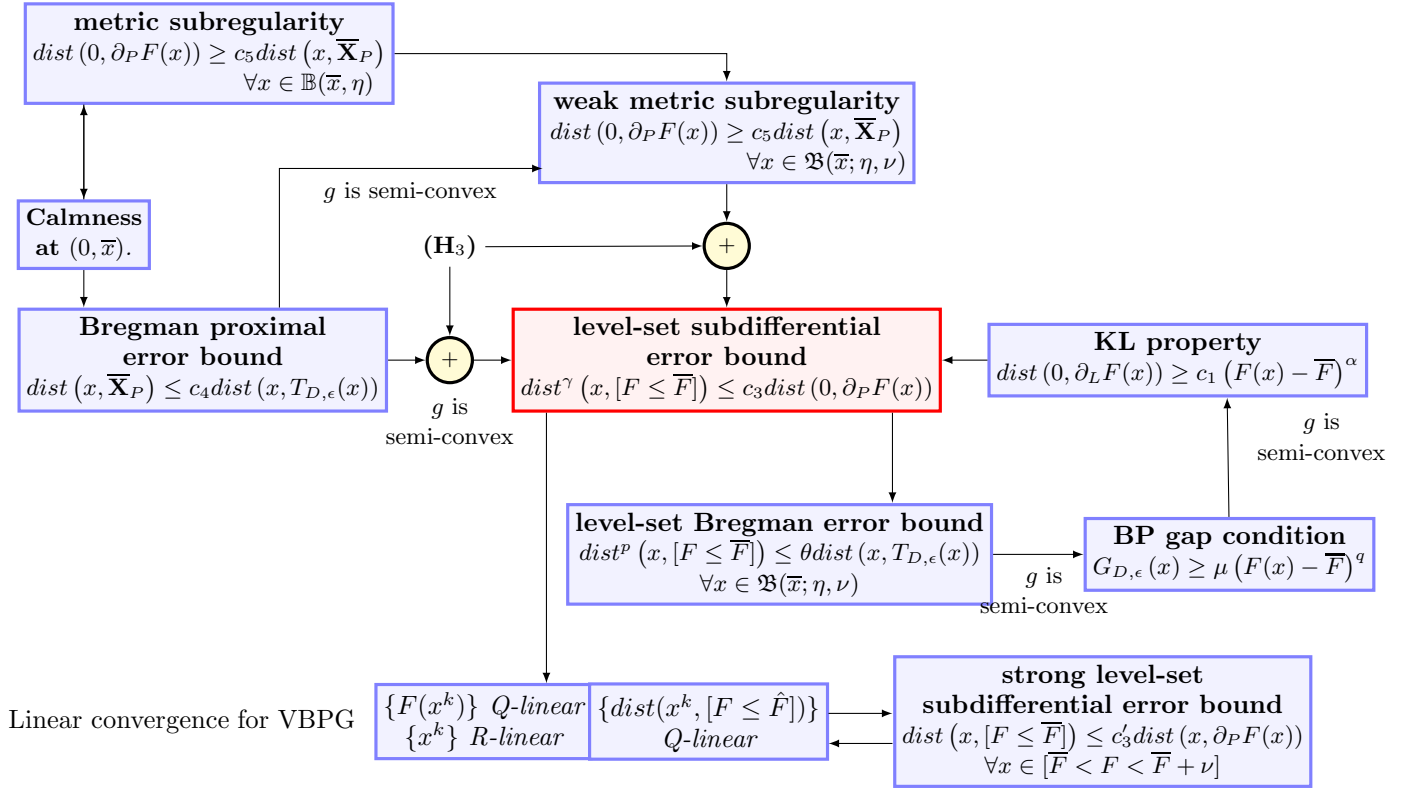
\subsubsection{Linear convergence of various algorithms under second type error bounds}
{\bf Application 7.3: Linear convergence under proximal-PL inequality and Bregman proximal gap}\\
\cite{Schmidt2016} proposes the concept of proximal-PL inequality for solving problem (P) where $F$ is invex function, $g$ is convex, i.e., there is $\mu>0$ such that the following inequality holds:
$$\frac{1}{2}D_g(x,L)\geq\mu\left(F(x)-F^*\right).$$
where $F^*$ is the global minimum value  and
$$D_g(x,\alpha)=-2\alpha\min_{y\in\RR^n}\left[\langle\nabla f(x),y-x\rangle+\frac{\alpha}{2}\|y-x\|^2+g(y)-g(x)\right],$$
which is a global version of Bregman proximal gap function with $D^k(x,y)=\frac{\|x-y\|^2}{2}$. \cite{Schmidt2016} proves the sequence $\{F(x^k)\}$ generated by PG method with a step size of $1/L$ linearly converges to $F^*$ under proximal-PL inequality. For the fully nonconvex case, Theorem~\ref{theo1}  shows the Q-linear convergence of $\{F(x^k)\}$ and the R-linear convergence of $\{x^k\}$ under the Bregman proximal gap condition, which is weaker than the proximal-PL inequality. Observe that the proximal PL inequality implies that every critical point achieves an optimum $F^*$, and the strong level-set subdifferential error bound condition holds. If $g$ is semi-convex, by Theorem~\ref{theo:n-s} the proximal PL inequality is also a necessary condition for linear convergence in the sense of~\eqref{x-Q-linear-2}.

{\bf Application 7.4: Linear convergence  under KL property}\\
Various variable metric proximal gradient methods are provided in following algorithms for problem (P)
$$x^{k+1}\rightarrow\min\langle\nabla f(x^k),x-x^k\rangle+g(x)+\frac{1}{2}\|x-x^k\|_{B_k}^2,$$
where $B_k$ is positive definite matrix.\\
The extrapolation and line-search techniques may be combined with the standard VMPG. For full convex problem (P), the converge rate of $O(1/k^2)$ of $\{F(x^k)\}$ is provided. Recently, E. Chonzennx et. al. \cite{Chouzenous2014} proposed an inexact version of VMPG algorithm for problem (P) where $g$ is convex. And the authors also provided linear convergence of VMPG under KL property with exponent $\frac{1}{2}$. Noted that VMPG is the special case with $D^k=\frac{\|x-y\|_{B_k}^2}{2}$, Theorem~\ref{theo1} states that VMPG has the linear convergence for $\{x^k\}$ and $\{F(x^k)\}$ under level-set subdifferential EB condition. Moreover, the strong level-set subdifferential error bound condition on $[\overline{F}<F<\overline{F}+\nu]$ is necessary and sufficient for linear convergence in the sense of~\eqref{x-Q-linear}. Mention that if $g$ is semi-convex, level-set subdifferential EB condition with exponent $\gamma=1$ is equivalent to KL exponent $\frac{1}{2}$ condition.

\section{Sufficient conditions for the  level-set subdifferential error bound condition to hold on $\mathfrak{B}(\overline{x};\eta,\nu)$ with $\overline{x}\in\overline{\mathbf{X}}_P$}\label{sec:sufficient}
In this section, we examine sufficient conditions to guarantee level-set subdifferential error bound condition
at $\overline{x}\in\overline{\mathbf{X}}_P$ on $\mathfrak{B}(\overline{x};\eta,\nu)$, where $\overline{x}$ is a proximal critical point of $F=f+g$\\

First, we provide some new notions
on relaxed strong convexity of function $f$ on $\mathbb{B}(\overline{x};\eta)$.
Given $z\in\mathbb{B}(\overline{x};\eta)$, for brevity, we denote
$Proj_{\mathbb{B}(\overline{x};\eta)\cap\overline{\mathbf{X}}_P}(z)$ by $\overline{z}_p$.
The following notations can be viewed as the local version for that in H. Karimi et al's and I. Necoara et al's paper~\cite{Schmidt2016},\cite{Necoara2018} respectively.
\begin{itemize}
\item[1.] Local strong-convexity (LSC) on $\mathbb{B}(\overline{x};\eta)$:
    $$f(y)\geq f(x)+\langle\nabla f(x),y-x\rangle+\frac{\mu}{2}\|y-x\|^2,\quad\forall x,y\in\mathbb{B}(\overline{x};\eta).$$
\item[2.] Local essentially-strong-convexity at $\overline{x}_p$ (LESC) on $\mathbb{B}(\overline{x};\eta)$:
    \begin{eqnarray}
    f(y)\geq f(x)+\langle\nabla f(x),y-x\rangle+\frac{\mu}{2}\|y-x\|^2,
    \quad\forall x,y\in\mathbb{B}(\overline{x};\eta)\quad\mbox{with}\quad\overline{x}_p=\overline{y}_p.
    \end{eqnarray}
\item[3.] Local weak- strong-convexity at $\overline{x}_p$ (LWSC) on $\mathbb{B}(\overline{x};\eta)$:
    \begin{eqnarray}
    f(\overline{x}_p)\geq f(x)+\langle\nabla f(x),\overline{x}_p-x\rangle+\frac{\mu}{2}\|\overline{x}_p-x\|^2,
    \quad\forall x\in\mathbb{B}(\overline{x};\eta).
    \end{eqnarray}
\item[4.] Local quadratic-gradient-growth (LQGG) at $\overline{x}_p$ on $\mathbb{B}(\overline{x};\eta)$:
    \begin{eqnarray}
    \langle\nabla f(x)-\nabla f(\overline{x}_p), x-\overline{x}_p\rangle\geq\mu\|\overline{x}_p-x\|^2,
    \quad\forall x\in\mathbb{B}(\overline{x};\eta).
    \end{eqnarray}
\end{itemize}
For the case $g=0$, the following two notions are introduced.
\begin{itemize}
\item[5.] Local restricted secant inequality (LRSI):
\begin{eqnarray}
\langle\nabla f(x),x-\overline{x}_p\rangle\geq\mu\|x-\overline{x}_p\|^2,\quad\forall x\in\mathbb{B}(\overline{x};\eta).
\end{eqnarray}
\item[6.] Local Polyak-{\L}ojasiewicz (LPL) inequality:
    \begin{eqnarray}
    \frac{1}{2}\|\nabla f(x)\|^2\geq\mu\left(f(x)-f(\overline{x})\right),\quad\forall x\in\mathbb{B}(\overline{x};\eta).
    \end{eqnarray}
\end{itemize}
It's easy to show that the following implications hold for the function $f$ on $\mathbb{B}(\overline{x};\eta)$.
$$(LSC)\Rightarrow(LESC)\Rightarrow(LWSC).$$
For the case $g=0$, the assumptions of $\nabla f(\overline{x}_p)=0$ and the LQGG  reduce to the local restricted secant inequality (LRSI). So  we have:
$$(LWSC)\Rightarrow(LRSI)\Rightarrow(LPL)\qquad\mbox{(if $g=0$)}.$$
Along with Assumption 3,  we can establish the level-set subdifferential error bound for $F$, whenever $g$ is prox-regular.
\begin{proposition}[Sufficient conditions for weak metric subregularity]\label{prop:5.1}
Suppose $\overline{x}\in\overline{X}_P$, $g$ is uniformly prox-regular around $\bar x\in \mathbf{dom}~g$ with modulus $\rho$. If  one of the following conditions holds
\begin{itemize}
\item[{\rm(i)}] $f$ is local weak strongly convex at $\overline{x}_p$ (LWSC) with modulus $\mu$ and $\mu>\rho$ on $\mathbb{B}(\overline{x};\eta)$.
\item[{\rm(ii)}] $f$ satisfies local quadratic gradient growth condition at $\overline{x}_p$ (LQGG) with modulus $\mu$ and $\mu>\rho$ on $\mathbb{B}(\overline{x};\eta)$,
\end{itemize}
then $F$ satisfies the weak metric subregularity condition at $\overline{x}$; that is,
$$dist(0,\partial_P F(x))\geq\frac{(\mu-\rho )}{2}dist(x,\overline{\mathbf{X}}_p),\quad\forall x\in\mathfrak{B}(\overline{x};\eta,\nu).$$
\end{proposition}
\begin{proof}
{\rm(i):} If $f$ is LWSC at $\overline{x}_p$ on $\mathbb{B}(\overline{x};\eta)$, then we have
\begin{equation}\label{eq:s-1}
f(\overline{x}_p)\geq f(x)+\langle\nabla f(x),\overline{x}_p-x\rangle+\frac{\mu}{2}\|\overline{x}_p-x\|^2.
\end{equation}
Since $g$ is uniformly prox-regular around $\bar x$ with $\rho$, then $\partial_Pg(x)=\partial_Lg(x)$ and
\begin{equation}\label{eq:s-2}
g(\overline{x}_p)\geq g(x)+\langle\xi,\overline{x}_p-x\rangle-\frac{\rho}{2}\|\overline{x}_p-x\|^2,\quad\forall \xi\in\partial_P g(x).
\end{equation}
Adding inequalities~\eqref{eq:s-1} and~\eqref{eq:s-2}, we obtain
$$F(\overline{x}_p)=F(\overline{x})=F_{\zeta}\geq F(x)+\langle\nabla f(x)+\xi,\overline{x}_p-x\rangle+\frac{(\mu-\rho )}{2}\|\overline{x}_p-x\|^2.$$
and
$$\langle\nabla f(x)+\xi,x-\overline{x}_p\rangle\geq\frac{(\mu-\rho )}{2}\|\overline{x}_p-x\|^2,\quad\forall\xi\in\partial_P g(x),\quad\forall x\in\mathbb{B}(\overline{x};\eta).$$
Using Cauchy-Schwartz on above inequality, we conclude
$$dist(0,\partial_PF(x))\geq\frac{(\mu-\rho )}{2}\|\overline{x}_p-x\|\geq\frac{(\mu-\rho )}{2}dist(x,\overline{\mathbf{X}}_P),\quad\forall x\in\mathbb{B}(\overline{x};\eta),$$
which yields the desired results.\\
{\rm(ii):} If $f$ is LQGG at $\overline{x}_p$ on $\mathbb{B}(\overline{x};\eta)$, then we have
$$\langle\nabla f(x)-\nabla f(\overline{x}_p),x-\overline{x}_p\rangle\geq\mu\|\overline{x}_p-x\|^2,\quad\forall x\in\mathbb{B}(\overline{x};\eta).$$
Since $g$ is semi-convex, we have
$$\langle u-v,x-\overline{x}_p\rangle\geq-\rho\|x-\overline{x}_p\|^2,\quad\forall u\in\partial_Pg(x),\quad\forall v\in\partial_Pg(\overline{x}_p).$$
Adding the above two inequalities for $x\in\mathbb{B}(\overline{x};\eta)$, we obtain
$$\langle(\nabla f(x)+u)-(\nabla f(\overline{x}_p)+v),x-\overline{x}_p\rangle\geq\frac{\mu}{2}\|\overline{x}_p-x\|^2.$$
Since $\overline{x}_p$ is a proximal critical point,  $0=\nabla f(\overline{x}_p)+v$ for some $v\in\partial_P g(\overline{x}_p)$.
With this choice of $v$, the last above inequality yields
$$\langle\nabla f(x)+u,x-\overline{x}_p\rangle\geq\frac{(\mu-\rho )}{2}\|\overline{x}_p-x\|^2\geq\frac{(\mu-\rho )}{2}dist(x,\overline{\mathbf{X}}_P),\quad\forall u\in\partial_P g(x),\quad\forall x\in\mathbb{B}(\overline{x};\eta).$$
This is enough for the proof of proposition.
\end{proof}
If we take $D^k(x,y)=\frac{\|x-y\|^2}{2}$, $\epsilon^k=\epsilon$, then VPBG is the proximal gradient (PG) method, the subproblem (AP$^k$) becomes to
\begin{equation}
x^{k+1}=Prox_{g}^{\epsilon}\left(x^k-\epsilon\nabla f(x^k)\right)
\end{equation}
where $Prox_{g}^{\epsilon}(y)=\arg\min\limits_{x\in\RR^n}\{g(x)+\frac{1}{2\epsilon}\|x-y\|^2\}$.\\
If $Prox_{g}^{\epsilon}(y)$ is a single-valued map and we have
\begin{equation}
\overline{x}\in\overline{\mathbf{X}}_P\Leftrightarrow 0\in\nabla f(\overline{x})+\partial_P g(\overline{x})\Leftrightarrow\overline{x}=Prox_{g}^{\epsilon}\left(\overline{x}-\epsilon\nabla f(\overline{x})\right).
\end{equation}
The following proposition present a sufficient conditions for Bregman proximal error bound.
\begin{definition}[Luo-Tseng error bound~\cite{Tseng2009}]
We say the Luo-Tseng error bound holds if any $\xi\geq\inf_{x\in\RR^n} F(x)$, there exists constant $c_6>0$ and $\sigma>0$ such that
\begin{equation}
dist(x,\overline{\mathbf{X}}_P)\leq c_6\|x-Prox_g^{\epsilon}\left(x-\epsilon\nabla f(x)\right)\|
\end{equation}
whenever $F(x)\leq\xi$, $\|x-Prox_g^{\epsilon}\left(x-\epsilon\nabla f(x)\right)\|\leq\sigma$.
\end{definition}
\begin{proposition}\label{prop:LuoTseng}
The Luo-Tseng error bound condition implies the Bregman proximal error bound when $g$ is convex.
\end{proposition}
\begin{proof}
First note that by the  hypotheses, $T_{D,\epsilon}(x)={Prox}_g^{\epsilon}(x-\epsilon\nabla f(x))$ is single-valued and continuous in $x$. For $\sigma$ and $\xi>\inf F(x)$, there are $\eta>0$ and $\nu\in(0,+\infty)$ such that $\|x-T_{D,\epsilon}(x)\|<\sigma$ and $F(x)\leq\xi$, whenever $\|x-\overline{x}\|\leq\eta$ and $F(x)\leq\overline{F}+\nu$. Since the Luo-Tseng error bound condition holds at $\overline{x}$,  we have
\begin{eqnarray}\label{eq:41}
dist(x,\overline{\mathbf{X}}_P)\leq c_6\|x-T_{D,\epsilon}(x)\|,\quad\forall x\in\mathfrak{B}(\overline{x};\eta,\nu).
\end{eqnarray}
This shows that  the  Bregman proximal error bound holds at $\overline{x}$.
\end{proof}
Now we are ready to present the main results on sufficient conditions to guarantee that
 the level-set subdifferential error
bound  holds at $\overline{x}$ on $\mathfrak{B}(\overline{x},\eta,\nu)$, where $\overline{x}$ is an accumulation point of the sequence $\{x^k\}$ generated by VBPG.
\begin{theorem}[Sufficient conditions for the existence of a level-set subdifferential EB]
Consider problem (P). Suppose that Assumption~\ref{assump1} and Assumption~\ref{assump2} hold, and $\overline{x}\in\overline{\mathbf{X}}_P$. If one of following conditions hold, then $F$ satisfies the level-set subdifferential error bound condition at $\overline{x}$ on $\mathfrak{B}(\overline{x};\eta,\nu)$.
\begin{itemize}
\item[{\rm(i)}] $F=f+g$ satisfies the KL exponent at $\overline{x}$ on $\mathfrak{B}(\overline{x};\eta,\nu)$ at $\overline{x}$.
\item[{\rm(ii)}] $F=f+g$ satisfies Bregman proximal error bound condition, $g$ is semi-convex or $g$ is uniformly prox-regular around $\overline{x}$, $x\in\mathfrak{B}(\overline{x};\frac{\eta}{2},\frac{\nu}{N})$ with $N\geq\frac{2\overline{\epsilon}\nu}{m-\overline{\epsilon}L}/\left(\frac{\eta}{2}\right)^2$ satisfies Property (A) and Assumption {\bf(H$_3$)} holds.
\item[{\rm(iii)}] $F=f+g$ satisfies weak metric subregularity at $\overline{x}$ and Assumption {\bf(H$_3$)} holds.
\item[{\rm(iv)}] With $g=0$, $f=F$ satisfies the (LPL) inequality on $\mathbb{B}(\overline{x};\eta)$.
\end{itemize}
\end{theorem}
\begin{proof}
(i) See the results of Section~\ref{subsec:level-set error bounds}.\\
(ii) \& (iii) See Theorem~\ref{theorem:subregularity}.\\
(iv) For this case, The (LPL) inequality implies the KL property. The assertion follows from Theorem~\ref{kl}.
\end{proof}
\begin{remark}
 For the composite optimization problem (P),
if we consider the global solution $\mathbf{X}^*$ instead of  $\overline{\mathbf{X}}_P$, then Assumption~\ref{assump4} is automaticcally satisfied.  Weak metric subregularity and the Bregman proximal error bound imply the level-set subdifferential error bound.
\end{remark}
\begin{remark}
From the definition of a level-set subdifferential error bound, suppose that $\overline{x}$ is a critical point. If $x\in\mathfrak{B}(\overline{x};\eta,\nu)$ is also a critical point, then $0\in\partial_PF(x)$ and $dist\left(x,[F\leq\overline{F}]\right)=0$. This fact follows $F(x)\leq F(\overline{x})$, which implies Assumption~\ref{assump4} is a necessary condition for level-set subdifferential error bounds to hold. We mention that Assumption~\ref{assump4} is also necessary for KL property.
\end{remark}

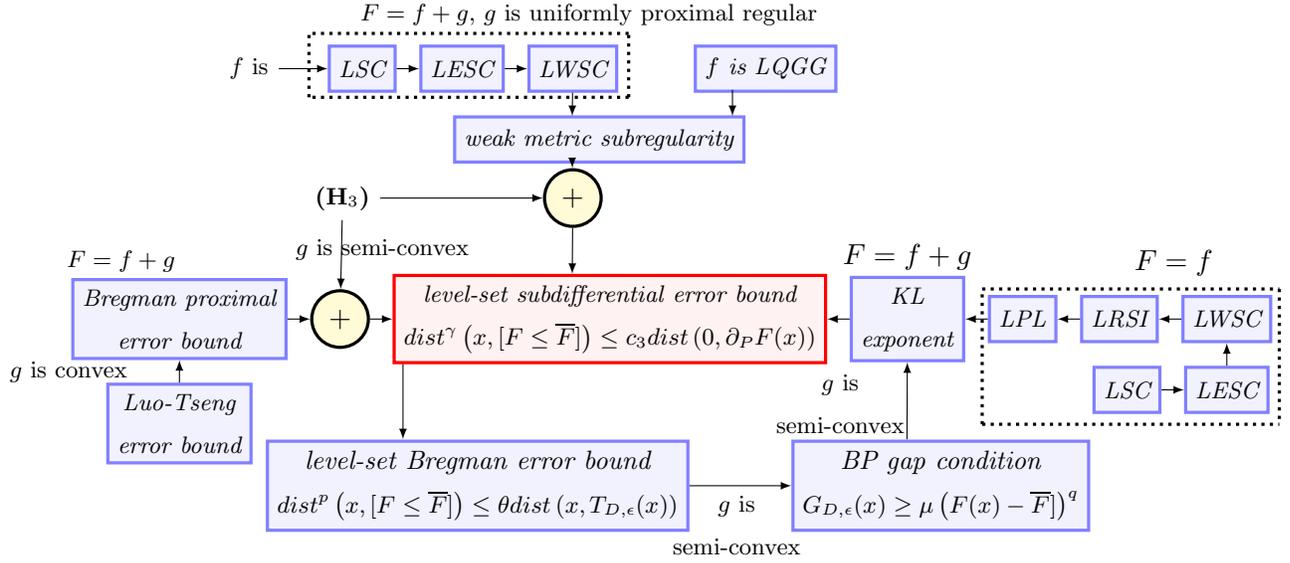
\begin{figure}
\begin{center}
\begin{tikzpicture}
[
>=latex,
node distance=3mm,
 ract/.style={draw=blue!50, fill=blue!5,rectangle,minimum size=6mm, very thick, font=\itshape, align=center},
 racc/.style={rectangle, align=center},
 ractm/.style={draw=red!100, fill=red!5,rectangle,minimum size=6mm, very thick, font=\itshape, align=center},
 cirl/.style={draw, fill=yellow!20,circle,   minimum size=6mm, very thick, font=\itshape, align=center},
 raco/.style={draw=green!500, fill=green!5,rectangle,rounded corners=2mm,  minimum size=6mm, very thick, font=\itshape, align=center},
 hv path/.style={to path={-| (\tikztotarget)}},
 vh path/.style={to path={|- (\tikztotarget)}},
 skip loop/.style={to path={-- ++(0,#1) -| (\tikztotarget)}},
 vskip loop/.style={to path={-- ++(#1,0) |- (\tikztotarget)}}]

        \node (a) [ractm]{\baselineskip=3pt\footnotesize level-set subdifferential error bound\\ \baselineskip=3pt\footnotesize$dist^{\gamma}\left(x,[F\leq\overline{F}]\right)\leq c_3{dist}\left(0,\partial_PF(x)\right)$};
        \node (b) [ract, below = of a, xshift=-50,yshift=-20]{\baselineskip=3pt\small level-set Bregman error bound\\ \baselineskip=3pt\footnotesize$dist^{p}\left(x,[F\leq\overline{F}]\right)\leq\theta dist\left(x,T_{D,\epsilon}(x)\right)$};
        \node (bb) [ract, right = of b, xshift=30]{\baselineskip=3pt\small BP gap condition\\ \baselineskip=3pt\footnotesize$G_{D,\epsilon}(x)\geq\mu\left(F(x)-\overline{F}]\right)^q$};
        \node (bbb) [racc, below= of b, xshift=98, yshift=26]{\baselineskip=3pt\footnotesize$g$ is\\
                                                              \baselineskip=3pt\footnotesize semi-convex};
        \node (bb1) [racc, above= of bb, xshift=-38,yshift=-10]{\baselineskip=0.1pt\footnotesize$g$ is\\
                                                     \baselineskip=0.1pt\footnotesize semi-convex};
        \node (d) [ract, right = of a]{\baselineskip=3pt\footnotesize KL\\ \baselineskip=3pt\footnotesize exponent};
        \node (ddd) [above = of d, yshift=-10]{$F=f+g$};
        \node (e) [ract, right = of d]{\baselineskip=3pt\footnotesize LPL};
        \node (f) [ract, right = of e]{\baselineskip=3pt\footnotesize LRSI};
        \node (g) [ract, right = of f]{\baselineskip=3pt\footnotesize LWSC};
        \node (gg) [above = of g, xshift=-20, yshift=-5]{$F=f$};
        \node (h) [ract, below = of g]{\baselineskip=3pt\footnotesize LESC};
        \node (i) [ract, left = of h]{\baselineskip=3pt\footnotesize LSC};
        \node (j) [ract, above = of a, xshift=-14, yshift=60]{\baselineskip=3pt\footnotesize LWSC};
        \node (jj) [ract, below = of j, xshift=10]{\baselineskip=3pt\footnotesize weak metric subregularity};
        \node (aa) [cirl, below = of j, yshift=-20]{$+$};
        \node (aba) [left = of aa, xshift=-53] {\baselineskip=3pt\footnotesize {\bf(H$_3$)}};
        \node (aba1) [racc, below= of aba, xshift=15, yshift=5]{\baselineskip=3pt\footnotesize$g$ is semi-convex};
        \node (k) [ract, left = of j]{\baselineskip=3pt\footnotesize LESC};
        \node (kk) [above= of k, xshift=48, yshift=-5]{\baselineskip=3pt\footnotesize $F=f+g$, $g$ is uniformly proximal regular};
        \node (l) [ract, left = of k]{\baselineskip=3pt\footnotesize LSC};
        \node (ll) [racc, left= of l, xshift=-10]{\baselineskip=3pt\footnotesize $f$ is};
        \node (m) [ract, right = of j, xshift=20]{\baselineskip=3pt\footnotesize $f$ is LQGG};
        \node (n) [cirl, left = of a]{$+$};
        \node (o) [ract, left  = of n]{\baselineskip=3pt\footnotesize Bregman proximal\\
                                       \baselineskip=3pt\footnotesize error bound};
        \node (o1) [ract, below = of o]{\baselineskip=3pt\footnotesize Luo-Tseng\\
                                        \baselineskip=3pt\footnotesize error bound};
        \node (o1o) [racc, above= of o1, xshift=-42, yshift=-12]{\baselineskip=3pt\footnotesize $g$ is convex};
        \node (oo) [above = of o, xshift=-22, yshift=-10]{\baselineskip=3pt\footnotesize $F=f+g$};
        \path (-2.75,-0.6) edge[->] (-2.75,-1.6)
              (3.95,-1.6) edge[->] (d)
              (b) edge[->] (bb)
              (aba) edge[->] (aa)
              (aa) edge[->] (-0.5,0.6)
              (j) edge[->] (-0.5,2.7)
              (-0.5,2.1) edge[->] (aa)
              (d) edge[->] (a)
              (e) edge[->] (d)
              (f) edge[->] (e)
              (g) edge[->] (f)
              (h) edge[->] (g)
              (i) edge[->] (h)
              (k) edge[->] (j)
              (l) edge[->] (k)
              (n) edge[->] (a)
              (aba) edge[->] (n)
              (o) edge[->] (n)
              (o1) edge[->] (o);
        \path (1.63,3.05) edge[->] (1.63,2.7)
              (ll) edge[->] (l);

        \draw[dotted,very thick] (4.95,-1.4)--(8.9,-1.4);
        \draw[dotted,very thick] (4.95,0.4)--(8.9,0.4);
        \draw[dotted,very thick] (4.95,-1.4)--(4.95,0.4);
        \draw[dotted,very thick] (8.9,0.4)--(8.9,-1.4);

        \draw[dotted,very thick] (0.25,2.95)--(-4,2.95);
        \draw[dotted,very thick] (0.25,3.8)--(-4,3.8);
        \draw[dotted,very thick] (0.25,2.95)--(0.25,3.8);
        \draw[dotted,very thick] (-4,2.95)--(-4,3.8);
\end{tikzpicture}
\caption{Sufficient conditions for the level-set subdifferential error bound \big{(}In this Figure, conditions of semiconvexity for $g$ can be replaced by uniformly proximal regularity with $\rho$ and $\eta$ and $x\in\mathfrak{B}(\overline{x};\frac{\eta}{2},\frac{\nu}{N})$ with $N\geq\frac{2\overline{\epsilon}\nu}{m-\overline{\epsilon}L}/\left(\frac{\eta}{2}\right)^2$ satisfies Property (A)\big{)}}\label{fig:2}
\end{center}
\end{figure}
\noindent {\bf Acknowledgments:} We are grateful for valuable feedbacks  from  Lei Zhao and Minghua Li   on earlier versions of this manuscript. Figures were computer-drawn by Lei Zhao, who also
in numerous  other ways generously gave  LaTex  technical support.

\end{document}